\documentclass[11pt,english]{amsart}
\usepackage[T1]{fontenc}
\usepackage[latin1]{inputenc}
\usepackage{verbatim}
\usepackage{amstext}
\usepackage{amsthm}
\usepackage{amssymb}
\usepackage{enumerate}

\usepackage{fullpage}
\usepackage{todonotes}

\usepackage{hyperref}
\usepackage{color}
\usepackage{cleveref}

\makeatletter
\numberwithin{equation}{section}
\numberwithin{figure}{section}
\theoremstyle{plain}
\newtheorem{thm}{\protect\theoremname}
\theoremstyle{plain}
\newtheorem{prop}[thm]{\protect\propositionname}
\theoremstyle{plain}
\newtheorem{cor}[thm]{\protect\corollaryname}
\theoremstyle{remark}
\newtheorem{rem}[thm]{\protect\remarkname}
\theoremstyle{definition}
\newtheorem{example}[thm]{\protect\examplename}
\theoremstyle{plain}
\newtheorem{lem}[thm]{\protect\lemmaname}
\theoremstyle{definition}
\newtheorem{defn}[thm]{\protect\definitionname}

\usepackage{tikz}
\usepackage{pgfplots}

\makeatother

\makeatother

\usepackage{babel}
\providecommand{\corollaryname}{Corollary}
\providecommand{\definitionname}{Definition}
\providecommand{\examplename}{Example}
\providecommand{\lemmaname}{Lemma}
\providecommand{\propositionname}{Proposition}
\providecommand{\remarkname}{Remark}
\providecommand{\theoremname}{Theorem}

\begin{document}

\title[Dynamics and operators on line arrangements]{Dynamical systems on some elliptic modular surfaces via operators on line arrangements}

\author{Lukas~K\"uhne}
\address{Lukas K\"uhne, Universit\"at Bielefeld, Fakult\"at f\"ur Mathematik, Bielefeld, Germany}
\email{lkuehne@math.uni-bielefeld.de}

\author{Xavier Roulleau}
\address{Xavier Roulleau, Universit\'e d'Angers, CNRS, LAREMA, SFR Math-STIC, F-49000 Angers, France}
\email{xavier.roulleau@univ-angers.fr}

\addtolength{\textwidth}{0mm}
\addtolength{\hoffset}{-0mm} 

\global\long\def\AA{\mathbb{A}}%
\global\long\def\CC{\mathbb{C}}%
 
\global\long\def\BB{\mathbb{B}}%
 
\global\long\def\PP{\mathbb{P}}%
 
\global\long\def\QQ{\mathbb{Q}}%
 
\global\long\def\RR{\mathbb{R}}%
 
\global\long\def\FF{\mathbb{F}}%

\global\long\def\DD{\mathbb{D}}%
 
\global\long\def\NN{\mathbb{N}}%
\global\long\def\ZZ{\mathbb{Z}}%
 
\global\long\def\HH{\mathbb{H}}%

\global\long\def\Gal{{\rm Gal}}%
\global\long\def\GG{\mathbb{G}}%
 
\global\long\def\UU{\mathbb{U}}%

\global\long\def\bA{\mathbf{A}}%

\global\long\def\kP{\mathfrak{P}}%
 
\global\long\def\kQ{\mathfrak{q}}%
 
\global\long\def\ka{\mathfrak{a}}%
\global\long\def\kP{\mathfrak{p}}%
\global\long\def\kn{\mathfrak{n}}%
\global\long\def\km{\mathfrak{m}}%

\global\long\def\cA{\mathfrak{\mathcal{A}}}%
\global\long\def\cB{\mathfrak{\mathcal{B}}}%
\global\long\def\cC{\mathfrak{\mathcal{C}}}%
\global\long\def\cD{\mathcal{D}}%
\global\long\def\cH{\mathcal{H}}%
\global\long\def\cK{\mathcal{K}}%

\global\long\def\cF{\mathcal{F}}%
 
\global\long\def\cI{\mathfrak{\mathcal{I}}}%
\global\long\def\cJ{\mathcal{J}}%

\global\long\def\cL{\mathcal{L}}%
\global\long\def\cM{\mathcal{M}}%
\global\long\def\cN{\mathcal{N}}%
\global\long\def\cO{\mathcal{O}}%
\global\long\def\cP{\mathcal{P}}%
\global\long\def\cQ{\mathcal{Q}}%

\global\long\def\cR{\mathcal{R}}%
\global\long\def\cS{\mathcal{S}}%
\global\long\def\cT{\mathcal{T}}%
\global\long\def\cW{\mathcal{W}}%

\global\long\def\kBS{\mathfrak{B}_{6}}%
\global\long\def\kR{\mathfrak{R}}%
\global\long\def\kU{\mathfrak{U}}%
\global\long\def\kUn{\mathfrak{U}_{9}}%
\global\long\def\ksU{\mathfrak{U}_{7}}%

\global\long\def\a{\alpha}%
 
\global\long\def\b{\beta}%
 
\global\long\def\d{\delta}%
 
\global\long\def\D{\Delta}%
 
\global\long\def\L{\Lambda}%
 
\global\long\def\g{\gamma}%
\global\long\def\om{\omega}%

\global\long\def\G{\Gamma}%
 
\global\long\def\d{\delta}%
 
\global\long\def\D{\Delta}%
 
\global\long\def\e{\varepsilon}%
 
\global\long\def\k{\kappa}%
 
\global\long\def\l{\lambda}%
 
\global\long\def\m{\mu}%

\global\long\def\o{\omega}%
 
\global\long\def\p{\pi}%
 
\global\long\def\P{\Pi}%
 
\global\long\def\s{\sigma}%

\global\long\def\S{\Sigma}%
 
\global\long\def\t{\theta}%
 
\global\long\def\T{\Theta}%
 
\global\long\def\f{\varphi}%
 
\global\long\def\ze{\zeta}%

\global\long\def\deg{{\rm deg}}%
 
\global\long\def\det{{\rm det}}%

\global\long\def\Dem{Proof: }%
 
\global\long\def\ker{{\rm Ker}}%
 
\global\long\def\im{{\rm Im}}%
 
\global\long\def\rk{{\rm rk}}%
 
\global\long\def\car{{\rm car}}%
\global\long\def\fix{{\rm Fix( }}%

\global\long\def\card{{\rm Card }}%
 
\global\long\def\codim{{\rm codim}}%
 
\global\long\def\coker{{\rm Coker}}%

\global\long\def\pgcd{{\rm pgcd}}%
 
\global\long\def\ppcm{{\rm ppcm}}%
 
\global\long\def\la{\langle}%
 
\global\long\def\ra{\rangle}%

\global\long\def\Alb{{\rm Alb}}%
 
\global\long\def\Jac{{\rm Jac}}%
 
\global\long\def\Disc{{\rm Disc}}%
 
\global\long\def\Tr{{\rm Tr}}%
 
\global\long\def\Nr{{\rm Nr}}%

\global\long\def\NS{{\rm NS}}%
 
\global\long\def\Pic{{\rm Pic}}%

\global\long\def\Km{{\rm Km}}%
\global\long\def\rk{{\rm rk}}%
\global\long\def\Hom{{\rm Hom}}%
 
\global\long\def\End{{\rm End}}%
 
\global\long\def\aut{{\rm Aut}}%
 
\global\long\def\SSm{{\rm S}}%

\global\long\def\psl{{\rm PSL}}%
 
\global\long\def\cu{{\rm (-2)}}%
 
\global\long\def\mod{{\rm \,mod\,}}%
 
\global\long\def\cros{{\rm Cross}}%
 
\global\long\def\nt{z_{o}}%

\global\long\def\co{\mathfrak{\mathcal{C}}_{0}}%

\global\long\def\ldt{\Lambda_{\{2\},\{3\}}}%
 
\global\long\def\ltd{\Lambda_{\{3\},\{2\}}}%
\global\long\def\lldt{\lambda_{\{2\},\{3\}}}%

\global\long\def\ldq{\Lambda_{\{2\},\{4\}}}%
 
\global\long\def\lldq{\lambda_{\{2\},\{4\}}}%
 
\begin{abstract}
This paper further studies the matroid realization space of a specific deformation of the regular $n$-gon with 
its lines of symmetry.
Recently, we obtained that these particular realization spaces are birational to the elliptic modular 
surfaces $\Xi_{1}(n)$ over the modular curve $X_1(n)$. 
Here, we focus on the peculiar cases when $n=7,8$ in more detail.
We obtain concrete quartic surfaces in~$\PP^3$ equipped with a 
dominant rational self-map stemming from an operator on line 
arrangements, which yields K3 surfaces with a dynamical system 
that is semi-conjugated to the plane.
\end{abstract}

\maketitle

\section{Introduction}

A \emph{line arrangement} $\cC=\ell_{1}+\dots+\ell_{k}$ is a finite union
of lines $\ell_{j}$ in the projective plane $\PP^{2}$. Line arrangements
are ubiquitous objects studied in various fields such as topology,
algebra, algebraic geometry, see for instance~\cite{Suciu,Yoshinaga} for two surveys.
In \cite{OSO}, the second author described a number of operators
acting on line arrangements: if $\mathfrak{n,m}$ are sets of integers
at least $2$, the operator $\L_{\mathfrak{m,n}}$ associates to a line
arrangement~$\cC$ the line arrangement $\L_{\mathfrak{m,n}}(\cC)$
which is the union of the lines that contain $n\in\mathfrak{n}$ points
among the $m$-points of $\cC$, for $m\in\mathfrak{m}$ (recall that
an $m$-point of $\cC$ is a point where exactly $m$ lines of $\cC$
meet). For example $\L_{\{2\},\{3\}}(\cC)$ is the union of the lines
that contain exactly three double points of $\cC$ (that line arrangement might be empty).

A labeled line arrangement $\cC=(\ell_{1},\dots,\ell_{k})$ is a line
arrangement for which one fixes the order of the lines. The
configuration of a labeled line arrangement $\cC$ is described by its associated
\emph{matroid} $M=M(\cC)$. Conversely, given a matroid $M$ (a combinatorial
object), one can look at line arrangements $\cC$ for which $M(\cC)=M$.
When such a $\cC$ exists, one says that $\cC$ is a realization of
$M$. Let us denote by $\cR=\cR(M)$ the moduli space of realizations
of $M$: a point of $\cR$ is the orbit under the action of the projective general linear group $\text{PGL}_{3}$
of a realization of $M$. The space of all realizations of $M$ is
denoted by $\kU=\kU(M)$ and there is a natural quotient map $\kU\to\cR$.

In \cite{KR}, we constructed a realizable matroid
$M_{n}$  for any $n\geq7$ that is based on the regular $n$-gon.
Interestingly, there exists an operator $\L$ among the ones we
described above (for example if $n=2k+1$ is odd, then $\L=\L_{\{2\},\{k\}}$)
which acts non-trivially on $\kU(M_n)$:
 if $\cC$ is a (generic) realization of $M_{n},$ then $\L(\cC)$ is also a realization of $M_{n}$. We obtain in that
way a dominant self-rational map $\l$ on the realization space $\cR_{n}=\cR(M_{n})$.

The main result of \cite{KR} establishes that the realization space $\cR_{n}$ is an
open dense sub-scheme of the \emph{elliptic modular surface} $\Xi_{1}(n)$, a well-studied surface, see e.g. Shioda's paper~\cite{Shioda}.
Recall that this surface $\Xi_{1}(n)$ parametrizes (up to isomorphisms)
triples $(E,t,p)$ of an elliptic curve and points $t,p$ on $E$
such that $t$ has order $n$. The modular curve $X_{1}(n)$ parametrizes
(up to isomorphisms) pairs $(E,t)$, where $E,t$ are as above. The
map $(E,t,p)\to(E,t)$ defines an elliptic fibration on $\Xi_{1}(n)$,
with fiber over the point $(E,t)$ isomorphic to $E$. For any integer
$m$, there is a natural multiplication by $m$ rational map of the
elliptic surface $\Xi_{1}(n)$. We obtain in \cite{KR} that, through
the identification of $\cR_{n}$ as an open subscheme of $\Xi_{1}(n)$,
the rational self map $\l$ induced by $\L$ is the multiplication
by $-2$ map acting on $\Xi_{1}(n)$, in particular $\l$ has degree
$4$. 

The aim of the present paper is to study the peculiar cases
when $n=7,8$ in more detail.
In particular, we give another proof that the surface $\cR_{n}$
is an open dense subscheme of $\Xi_{1}(n),$ and the degree of $\l$
is $4$ in these cases.
From now on assume $n\in\{7,8\}$; in those cases, we obtain 
(singular) models of $\Xi_{1}(n)$ as quartic
surfaces in $\PP^{3}$. There is a natural section $\cR_{n}\to\kU_{n}=\kU(M_{n})$
of the quotient map $\kU_{n}\to\cR_{n}$, so that one may consider
$\cR_{n}$ as contained in $\kU_{n}$, and therefore one may consider
a class as a realization of $M_{n}$. Using that fact, we are able
to give explicit polynomials for the action $\l=\l(n)$ of $\L=\L(n)$ on
$\cR_n\subset\PP^{3}$.

Recall that a dynamical system is a pair $(X,\l)$ of a variety $X$ and a dominant rational map $\l:X\to X$.
A dynamical
system $(X,\l)$ is called \emph{semi-conjugated} to a dynamical system $(Y,\mu)$ if there exists
a generically finite rational dominant map $\pi:X\to Y$ such that
$\pi\circ\l=\mu\circ\pi$.
A principal result of this article is the following.
\begin{thm}
For $n\in\{7,8\}$, the dynamical system $(\cR_{n},\l)$ is semi-conjugated
to $(\PP^{2},F)$ where $F:\PP^{2}\dashrightarrow\PP^{2}$ is an explicitly
described rational self map; the dominant rational map $\pi:\cR_{n}\to\PP^{2}$
such that $\pi\circ\l=F\circ\pi$ is a double cover of $\PP^{2}$
branched along a sextic curve.
\end{thm}

The surfaces $\Xi_{1}(7),\,\Xi_{1}(8)$ are K3 surfaces; to our knowledge
these are the first examples of a degree~$>1$ dynamical system on a K3 surfaces
that is semi-conjugated to the plane. 

Let us describe the structure of this paper and some further results.
In Section \ref{sec:Action-of-}, we start by describing the line operators $\L$
and general results on matroids. In Subsection \ref{subsecTheoretical-SConj},
we study under which conditions a K3 surface which is the double cover
of the $\PP^2$ may be semi-conjugated to $\PP^2$.
Subsequently,
we study the case $n=7$ in Section \ref{sec:The-heptagon}:
we start by recalling the definition of the matroid $M_{7}$ and then show that
$\L_{\{2\},\{3\}}$ induces a rational self map $\lambda_{\{2\},\{3\}}$
on the quartic surface $\cR_{7}\subset\PP^{3}$.
We then compute the
degree of $\lambda_{\{2\},\{3\}}$ and prove that $\cR_{7}$
is an open subset of the elliptic modular surface~$\Xi_{1}(7)$. The
automorphism group of the matroid $M_{7}$ is the order $42$ Frobenius
group. There is a natural action of that group on the surface $\cR_{7}$. 
We show that this action is faithful.
The quotient surface $\cR_{7}/\aut(M_{7})$ is the moduli space for
unlabeled line arrangements coming from realizations of $M_{7}$:
we obtain that this is a rational surface.  In Subsection
\ref{subsec:Practical-SemiConj}, we describe explicitly the semi-conjugacy
of $\cR_{7}$ (or equivalently $\Xi_1(7)$) with $\PP^{2}$. The branch loci
of the double cover $\Xi_1(7)\to \PP^2$ is the union of a line and a 
singular quintic curve which we describe. Section \ref{sec:The-Octagon-and}
follows a similar pattern for the case $n=8$. In that case, the branch loci of the 
double cover $\Xi_1(8)\to \PP^2$ is union of a conic and a singular
 quartic curve. We moreover describe some $3$-periodic 
line arrangements for $\Lambda$; their classes are fixed points for the 
action of $\lambda$ on $\cR_8$. 

We remark that for $n=9$, one may similarly obtain that $\cR_9$ (contained 
as a sextic surface in $\PP^3$) is birational to $\Xi_1(9)$. That 
elliptic surface is no longer a K3 surface and we could not find 
a semi-conjugacy with the plane. 

Computations in this paper are based on Magma \cite{Magma} and 
\texttt{OSCAR} \cite{Oscar}. In the arXiv ancillary file of this paper 
are some datas related to these computations. 

\textbf{Acknowledgments} The authors are grateful to Bert van Geemen
and to Keiji Oguiso for interesting discussions.
Moreover, the authors thank the anonymous referees for their helpful remarks which helped to improve the article.
LK is supported by the SFB-TRR 358 -- 491392403 ``Integral Structures in Geometry and Representation Theory'' (DFG). XR is thankful to Max Planck Institute for Mathematics in Bonn for its hospitality and financial support. 
XR is also supported by the Centre Henri Lebesgue (ANR-11-LABX-0020-01).

\section{\label{sec:Action-of-}Notations and definitions}

Throughout this article we assume to be working over the field $\CC$.

\subsection{\label{subsec:Line-arrangement-and}Line arrangements and the operator $\Lambda_{\mathfrak{n,m}}$}

A line arrangement $\cC=\ell_{1}+\dots+\ell_{n}$ is a union of 
finitely many distinct lines in $\PP^{2}$. A labeled line arrangement
$\cC=(\ell_{1},\dots,\ell_{n})$ is a line arrangement
with a numbering of the lines. We sometime put a superscript
$^{\ell}$ (resp.~$^{u}$) when we want to emphasize that an arrangement
or related objects has (resp. does not have) a labeling. 

For an integer $k\geq2$, a $k$-point of the line arrangement $\cC$
is a point where exactly $k$ lines of $\cC$ meet. 
As in \cite{OSO},
for a subset $\mathfrak{n}$ of integers at least $2$,
let us denote by $\cP_{\mathfrak{n}}(\cC)$ the set of $k$-points
of $\cC$ for all $k\in\mathfrak{n}$. 
We denote by $t_{k}=t_{k}(\cC)=|\cP_{\{k\}}(\cC)|$ the number of $k$-points of $\cC$.
For a finite set of point $\cP$
in $\PP^{2}$ and $\mathfrak{n}$ as above, we denote by $\cL_{\mathfrak{n}}(\cP)$
the set of lines which contain exactly $n$ points in $\cP$ for some $n\in\mathfrak{n}$.

For subsets $\mathfrak{n,m}$ of integers at least $2$,
let us denote by $\Lambda_{\mathfrak{n,m}}(\cC)=\cL_{\mathfrak{m}}\circ\cP_{\mathfrak{n}}(\cC)$
the line arrangement that contains all
lines of $\PP^{2}$ containing exactly $m$ points of $\cP_{\mathfrak{n}}(\cC)$ for $m\in \mathfrak{m}$.
For example $\L_{\{2\},\{3,4\}}(\cC)$ is the union of the lines that
contain three or four double points of $\cC$. 
The arrangement could be the empty arrangement if no such lines exists.

\subsection{\label{sec:Matroids}Matroids and the period map of the moduli of
a matroid}

A matroid is a fundamental and actively studied object in combinatorics.
Matroids generalize linear dependency in vector spaces as well as forests in graphs.
See e.g.~\cite{Oxley} for a comprehensive treatment of matroids.
We just briefly mention a few concepts about matroids that are relevant for this article.

A \emph{matroid} is a pair $M=(E,\cB)$, where $E$ is a finite \emph{ground set}
of elements called atoms and $\cB$ is a nonempty collection of subsets of
$E$, called \emph{bases}, satisfying an exchange property reminiscent from linear algebra.

The prime examples of matroids arise by choosing a finite set $E$ of vectors in a vector space and declaring the maximal linearly independent subsets of $E$ as bases. In our case we obtain matroids through line arrangements:
If $\cC=(\ell_{1},\dots,\ell_{m})$
is a labeled line arrangement, the subsets $\{i,j,k\}\subseteq\{1,\dots,m\}$
such that the lines $\ell_{i},\ell_{j},\ell_{k}$ meet in three distinct
points are the bases of a matroid $M(\cC)$ over the set $\{1,\dots,m\}$.
We say that $M(\cC)$ is the matroid associated to $\cC$.

We denote by $\aut(M)$ the \emph{automorphism group} of the matroid $M$,
i.e., the set of isomorphisms from $M$ to $M$.

A \emph{realization} (over some field) of a matroid $M=(E,\cB)$ is a converse
operation to the association $\cC\to M(\cC)$: it is a
$3\times m$-matrix with non-zero columns $C_{1},\dots,C_{m}$,
which are considered up to a multiplication by a scalar (thus as point
in the projective plane) such that a subset $\{i_{1},i_{2},i_{3}\}$
of $E$ of size $3$ is a basis if and only if the $3\times 3$ minor $|C_{i_{1}},C_{i_{2}},C_{i_{3}}|$
is nonzero.
We denote by $\ell_{i}$ the line with normal vector the
point $C_{i}\in\PP^{2}$.

If $\cC=(\ell_{1},\dots,\ell_{m})$ is a realization of $M$ and $\g\in PGL_{3}$,
then $(\g\ell_{1},\dots,\g\ell_{m})$ is another realization of $M$; we denote
by $[\cC]$ the orbit of $\cC$ under that action of $PGL_{3}$. The
\emph{moduli space} $\cR(M)$ of realizations of $M$ parametrizes the orbits
$[\cC]$ of realizations.
A more detailed introduction to these moduli spaces together with a description
of a software package in \texttt{OSCAR} that can compute these spaces is given in~\cite{Oscar}.

In this article, we always assume that each
subset of three elements of the first four atoms is a basis (otherwise,
we replace $M$ by a matroid isomorphic to it). Then in the moduli
space $\cR(M)$, one can always map the first four vectors of $\cC\in[\cC]$
to a fixed projective basis, so that each element $[\cC]$ of $\cR(M)$
has a canonical representative, which we will identify with $[\cC]$.

A useful tool for the computations related to the moduli space $\cR=\cR(M)$
of realizations of a matroid $M$ is what we call the period map: 
Let us denote by $\kU=\kU(M)$ the scheme of all realizations of $M$
in $\PP^{2}$. By analogy with similar objects, we call the quotient
map 
\[
\mathfrak{q}:\kU(M)\to\cR(M)
\]
the period map; a point $c$ of $\cR=\cR(M)$ is the class $c=[\cC]$
of a realization $\cC$.
Once a basis is fixed,
each class $c$ has a unique representative $\cC_{0}$ and we can
(and we will) identify $c$ with that representative. 

It often occurs that $\cR$ is embedded in a space $\mathbb{S}=\mathbb{S}(y_{1},\dots,y_{k})$
(affine or projective) of small dimension, like $\PP^{3}$. The
coordinates of the normal vectors $n^{(j)}=(n_{1}^{(j)}:n_{2}^{(j)}:n_{3}^{(j)})$
of $\cC_{0}$ are then polynomials $n_{1}^{(j)}=P_{1}^{(j)}(y),\dots,n_{3}^{(j)}=P_{3}^{(j)}(y)$
in the coordinates $y_{1},\dots,y_{k}$ of $\cR$ in $\mathbb{S}$. 

One often arrives at the natural question on computing the point $y=(y_{1},\dots,y_{k})$
in $\cR$ from the knowledge of the normal vectors $n$. In other
words, we need an explicit form of the period map $\mathfrak{q}$
as a map from $\kU$ to the scheme $\cR$ embedded in the space $\mathbb{S}$.
The answer to that problem are polynomials (or rational functions)
$Q_{1},\dots,Q_{k}$ in the coordinates of the normal vectors $n^{(1)},\dots,n^{(m)}$
etc.; here $m$ is the number of lines in an arrangement.

\subsection{\label{subsecTheoretical-SConj}Degree two K3 surfaces semi-conjugated
to the plane}

Let $C_{1}:Q_{1}=0$ be a sextic curve with at most ADE singularities,
so that the desingularization $X^{s}$ of the associated double cover
\[
X=\{y^{2}=Q_{1}(z_{1},z_{2},z_{3})\}\hookrightarrow\PP(3,1,1,1)
\]
is a K3 surface. Let $F:\PP^{2}\to\PP^{2}$ be a rational self-map
defined by coprime homogeneous polynomials $(F_{1},F_{2},F_{3})$
of degree $m$. Suppose that $F^{*}C_{1}=C_{1}+2D$, for an effective
divisor $D$; in algebraic terms, that means that we assume that
\[
Q_{1}(F_{1},F_{2},F_{3})=Q_{1}\cdot R^{2},
\]
for some polynomial $R$.
Then the following relation holds
\[
(y\,R(z))^{2}=Q_{1}(z)\,R(z)^{2}=Q_{1}(F_{1}(z),F_{2}(z),F_{3}(z)),
\]
where $z=(z_{1}:z_{2}:z_{3})\in\PP^{2}$. Hence, the rational map 
\[
\tilde{F}:(y;z)\dashrightarrow(y\,R(z);F_{1}(z):F_{2}(z):F_{3}(z))
\]
is a rational self-map acting on the K3 surface $X^{s}$.
Let $\pi:X\to\PP^{2}$ be the double cover map. The following diagram
\begin{equation}
\begin{array}{ccc}
X & \stackrel{\tilde{F}}{\dashrightarrow} & X\\
\pi\downarrow &  & \pi\downarrow\\
\PP^{2} & \stackrel{F}{\dashrightarrow} & \PP^{2}
\end{array}\label{eq:com-Diag}
\end{equation}
is commutative and, by analogy with other dynamical systems, we say
that the dynamical system $(X,\tilde{F})$ is \emph{semi-conjugated} to $(\PP^{2},F)$.
\begin{example}
Let $C$ be an irreducible curve of degree $6$ with $10$ nodes.
A Coble surface $Y$ is the blow-up of $\PP^{2}$ at the $10$ nodal
singularities of $C$. The group of birational transformations $G$
preserving $C$ is infinite, it is generated by Bertini involutions
centered at the nodal points of $C$. When $C$ is generic, the group
$G$ lifts to $Y$ and the elements of $G$ become automorphisms of
$Y$. The automorphism group $G\subset\aut(Y)$ preserves the pull-back
$C'$ of $C$, thus taking the double cover of $Y$ branched over
$C'$, one gets a smooth K3 surface $X$ and the group $G$ is in
fact the automorphism group of $X$ (see e.g. \cite{CD}). The surface
$X$ is also the minimal desingularization of the double cover branched
over $C$ and the diagram \eqref{eq:com-Diag} is commutative.
\end{example}

\section{\label{sec:The-heptagon}The heptagon}

\subsection{$\protect\ldt$ is a\label{subsec:rational-self} rational self-map
on $\protect\cR_{7}$ and $\protect\ksU$}

\subsubsection{Definition of the matroid $M_{7}$}

The matroid $M_{7}$ has $14$ atoms $1,\dots,7,1',\dots,7'$ and the bases are the
triples $\{a,b,c\}$ with $\{a,b\}\subset\{1,\dots,7\}$ and $c\in\{1',\dots,7'\}$
such that $a+b\neq 2c \mod 7$.
A sketch of $M_7$ is described in Figure \ref{fig:MATROID}, where the atoms $i\in \{1,\dots,7\}$ and $j\in \{1',\dots,7'\}$ correspond to the lines $\ell_i$ and $\ell_j'$, resp., and three lines form a basis if they do not meet in one point. Note that the central singularity of arrangement in Figure~\ref{fig:MATROID} is not part of the matroid and therefore removed.
\begin{figure}[h]
\begin{center}

\begin{tikzpicture}[scale=0.25]

\clip(-21.365,-15.33) rectangle (18.02,17.88);
\draw [line width=0.3mm,domain=-20.365:18.02] plot(\x,{(-10.302397111536783-1.8704694055762006*\x)/-2.34549444740409});
\draw [line width=0.3mm,domain=-20.365:18.02] plot(\x,{(-7.956902664132689-1.8704694055762001*\x)/2.345494447404089});
\draw [line width=0.3mm,domain=-20.365:18.02] plot(\x,{(-7.958578673150907--0.6675628018689426*\x)/2.9247837365454705});
\draw [line width=0.3mm,domain=-20.365:18.02] plot(\x,{(--10.883362409696378-0.6675628018689439*\x)/2.9247837365454705});
\draw [line width=0.3mm] (-3.,-15.33) -- (-3.,17.88);
\draw [line width=0.3mm,domain=-20.365:18.02] plot(\x,{(-9.00376595807958--2.702906603707257*\x)/1.3016512173526746});
\draw [line width=0.3mm,domain=-20.365:18.02] plot(\x,{(-10.305417175432254--2.702906603707257*\x)/-1.3016512173526738});
\draw [line width=0.3mm,color=blue,domain=-20.365:18.02] plot(\x,{(-4.226434953898144-0.*\x)/-8.452869907796291});
\draw [line width=0.3mm,color=blue,domain=-20.365:18.02] plot(\x,{(--9.502155104945505-8.240938811152406*\x)/17.11248576920136});
\draw [line width=0.3mm,color=blue,domain=-20.365:18.02] plot(\x,{(--2.3522804693791564-10.276282612990718*\x)/2.34549444740409});
\draw [line width=0.3mm,color=blue,domain=-20.365:18.02] plot(\x,{(-2.340052480306287-8.240938811152398*\x)/-6.571929401302233});
\draw [line width=0.3mm,color=blue,domain=-20.365:18.02] plot(\x,{(-4.2318769209959335--8.240938811152395*\x)/-6.571929401302231});
\draw [line width=0.3mm,color=blue,domain=-20.365:18.02] plot(\x,{(-0.006786021975052847--10.27628261299071*\x)/2.345494447404089});
\draw [line width=0.3mm,color=blue,domain=-20.365:18.02] plot(\x,{(--4.223414890002658--4.573376009283455*\x)/9.496713137847706});
\begin{scriptsize}
\draw [fill=black] (-3.,2.) circle (2.6mm);
\draw[color=black] (-3.8,2.99) node {$p_{23}$};

\draw [fill=black] (-3.,-1.) circle (2.6mm);
\draw[color=black] (-2.41,-0.1875) node {$p_{34}$};

\draw [fill=black] (-0.6545055525959113,-2.8704694055762) circle (2.6mm);
\draw[color=black] (-0.07,-2.0325) node {$p_{45}$};

\draw [fill=black] (2.2702781839495594,-2.2029066037072575) circle (2.6mm);
\draw[color=black] (2.855,-1.3575) node {$p_{56}$};

\draw [fill=black] (3.571929401302234,0.5) circle (2.6mm);
\draw[color=black] (4.99,1.0425) node {$p_{67}$};

\draw [fill=black] (2.2702781839495603,3.2029066037072567) circle (2.6mm);
\draw[color=black] (2.855,4.0425) node {$p_{17}$};

\draw [fill=black] (-0.65450555259591,3.8704694055762006) circle (2.6mm);
\draw[color=black] (-0.07,4.7175) node {$p_{12}$};
\draw[color=black] (-18,-10.99) node {$\ell_2$};
\draw[color=black] (14.35,-13.5) node {$\ell_4$};
\draw[color=black] (17,2.1325) node {$\ell_5$};
\draw[color=black] (-19.735,7.5) node {$\ell_1$};
\draw[color=black] (-2.1,-14.2) node {$\ell_3$};
\draw[color=black] (10.6,17.1825) node {$\ell_6$};
\draw[color=black] (-5.3,17.1825) node {$\ell_7$};

\draw [fill=black] (-4.880940506494056,0.5) circle (2.5mm);
\draw[color=black] (-4.9,1.5525) node {$p_{24}$};

\draw[color=blue] (-19.6,1.2525) node {$\ell_3'$};

\draw [fill=black] (-12.49671313784771,6.573376009283462) circle (2.5mm);
\draw[color=black] (-11.905,7.3275) node {$p_{14}$};

\draw [fill=black] (-12.496713137847706,-5.573376009283455) circle (2.5mm);
\draw[color=black] (-12.905,-4.8225) node {$p_{25}$};

\draw [fill=black] (-3.,-13.146752018566911) circle (2.5mm);
\draw[color=black] (-4.3,-12.9) node {$p_{36}$};

\draw [fill=black] (8.84220758525179,-10.443845414859652) circle (2.5mm);
\draw[color=black] (9.925,-9.6825) node {$p_{47}$};

\draw [fill=black] (14.112485769201351,0.5) circle (2.5mm);
\draw[color=black] (14.69,1.2525) node {$p_{15}$};

\draw [fill=black] (8.842207585251794,11.443845414859656) circle (2.5mm);
\draw[color=black] (8.125,12.1875) node {$p_{26}$};

\draw [fill=black] (-3.,14.146752018566918) circle (2.5mm);
\draw[color=black] (-1.71,14.1875) node {$p_{37}$};

\draw [fill=black] (-3.,4.405813207414516) circle (2.5mm);
\draw[color=black] (-2.41,5.1675) node {$p_{13}$};

\draw [fill=black] (-3.,-3.4058132074145138) circle (2.5mm);
\draw[color=black] (-3.99,-2.9625) node {$p_{35}$};

\draw [fill=black] (1.2264349538981438,-4.3704694055762) circle (2.5mm);
\draw[color=black] (0.0,-4.275) node {$p_{46}$};

\draw [fill=black] (4.615772631353649,2.667562801868942) circle (2.5mm);
\draw[color=black] (6.195,2.925) node {$p_{16}$};

\draw [fill=black] (1.2264349538981456,5.3704694055762) circle (2.5mm);
\draw[color=black] (1.8,6.6125) node {$p_{27}$};

\draw [fill=black] (4.615772631353649,-1.667562801868944) circle (2.5mm);
\draw[color=black] (5.195,-0.9075) node {$p_{57}$};

\draw[color=blue] (15.,-8.) node {$\ell_6'$};
\draw[color=blue] (2.04,-13.1825) node {$\ell_5'$};
\draw[color=blue] (-12.5,-14) node {$\ell_4'$};
\draw[color=blue] (-12.,17.1825) node {$\ell_2'$};
\draw[color=blue] (3.,17.1825) node {$\ell_1'$};
\draw[color=blue] (14.,8.2) node {$\ell_7'$};
\end{scriptsize}
\draw [fill=white] (0.15,0.5) circle (7mm);

\end{tikzpicture}

\end{center} 

\caption{\label{fig:MATROID}The matroid $M_{7}$ whose construction is based on the regular heptagon.}
\end{figure}

Let $\cA_{1}$ be a generic line arrangement realizing the matroid $M_{7}$. We 
write $\cA_{1}=\cC_{0}\cup\cC_{1}$ where $\cC_{0}$ are the first
seven lines and $\cC_{1}$ are the seven last ones. By the combinatorics
of the matroid $M_{7}$ and the genericity assumption, the property $\cC_{1}=\L_{\{2\},\{3\}}(\cC_{0})$
holds, and -- that will be important for us -- the image of $\cC_{0}$
by the operator $\ldt$ has a natural labeling: for any $j\in\{1,\dots,7\}$,
the six line arrangement 
\begin{equation}
H_{j}=\sum_{k\in\{1,\dots,7\},\,k\neq j}\ell_{k}\label{eq:Hexagons-1}
\end{equation}
is such that the line arrangement $\ldt(H_{j})$ is a unique line
$\ell_{j}'$, moreover: 
\[
\cC_{1}=(\ell_{1}',\dots,\ell_{7}').
\]
 Since $\cC_{1}=\L_{\{2\},\{3\}}(\cC_{0})$, to shorten our notations, we will often
speak of $\cC_{0}$ as a realization of $M_{7}$ instead of $\cC_{0}\cup\cC_{1}$.

The singularities of $\cC_{0}$ (resp. $\cC_{1}$) are $21$ double
points. The $21$ singularities on $\cC_{0}$ become the triple points
on $\cC_{0}\cup\cC_{1}$, moreover $t_{2}(\cC_{0}\cup\cC_{1})=28$.

\subsubsection{Equation of the quartic surface $Z_7$ and realization space of $M_{7}$}

Consider $Z_7$, the quartic surface in $\PP^{3}$  given by the equation {\footnotesize
\begin{equation}
y_{1}^{2}y_{2}^{2}+y_{1}^{2}y_{2}y_{3}-y_{1}y_{2}^{2}y_{3}-y_{1}y_{2}y_{3}^{2}-y_{1}^{2}y_{2}y_{4}-y_{1}y_{2}^{2}y_{4}+y_{1}y_{2}y_{3}y_{4}-y_{2}y_{3}^{2}y_{4}+y_{1}y_{2}y_{4}^{2}+y_{3}^{2}y_{4}^{2}=0.\label{eq:quarti-Z}
\end{equation}
}The eight singularities of $Z_7$ are of type $4A_{1}+A_{2}+3A_{3},$
at the points respectively{\footnotesize
\[
\begin{array}{c}
s_{1}=(0:0:0:1),s_{2}=(1:0:0:1),s_{3}=(0:0:1:0),s_{4}=(1:0:1:0),\\
s_{5}=(0:1:0:0),s_{6}=(0:1:0:1),s_{7}=(1:-1:1:0),s_{8}=(1:0:0:0).
\end{array}
\]
}The minimal desingularization of $Z_7$ is a K3 surface which we denote
by $Z_7^{s}$. Let $x_{1},x_{2},x_{3}$ be the coordinates on the affine
chart $y_{4}\neq0$. For a generic point $x=(x_{1},x_{2},x_{3})$
on the surface $Z_7$ in the chart $y_{4}\neq0$, let us define the 
 labeled arrangement of seven lines $\co=\co(x)$ with normal vectors the points
$p_{1},\dots,p_{7}$ respectively defined by{\footnotesize{}
\begin{equation}
\begin{array}{c}
(1:0:0),(0:1:0),(0:0:1),(-1:1:1)\\
(-x_{1}x_{2}^{2}-x_{1}x_{2}x_{3}+x_{1}x_{2}-x_{2}x_{3}+x_{3}\,:\,x_{1}x_{2}+x_{1}x_{3}-x_{1}\,:\,x_{2}-1)\\
(-x_{1}x_{2}^{2}-x_{1}x_{2}x_{3}+x_{1}x_{2}-x_{2}x_{3}+x_{3}\,:\,x_{1}x_{2}+x_{1}x_{3}-x_{1}+x_{2}^{2}\\
+x_{2}x_{3}-2x_{2}-x_{3}+1\,:\,x_{2}^{2}+x_{2}x_{3}-x_{2}-x_{3})\\
(-x_{1}x_{2}^{2}-x_{1}x_{2}x_{3}+x_{1}x_{2}+x_{3}^{2}\,:\,x_{1}x_{2}+x_{1}x_{3}-x_{1}-x_{2}x_{3}\\
-x_{3}^{2}+x_{3}\,:\,x_{2}^{2}+x_{2}x_{3}-x_{2}-x_{3}).
\end{array}\label{eq:of-mathcal-Co}
\end{equation}
}Let us also define the lines arrangement $\cC_{1}=\cC_{1}(x)$ with
normal vectors{\footnotesize{}
\begin{equation}
\begin{array}{c}
(-x_{1}x_{2}^{2}-x_{1}x_{2}x_{3}+x_{1}x_{2}+x_{3}^{2}\,:\,x_{1}x_{2}^{2}+2x_{1}x_{2}x_{3}-x_{1}x_{2}+x_{1}x_{3}^{2}\\
-x_{1}x_{3}-x_{2}^{2}x_{3}-2x_{2}x_{3}^{2}+x_{2}x_{3}-x_{3}^{3}+x_{3}^{2}\,:\,x_{2}^{2}+x_{2}x_{3}-x_{2}-x_{3}),\\
(-x_{1}x_{2}-x_{1}x_{3}+x_{1}\,:\,x_{1}x_{2}+x_{1}x_{3}-x_{1}\,:\,x_{2}-1),(-x_{2}:1:0),\\
(-x_{1}x_{2}^{3}-2x_{1}x_{2}^{2}x_{3}+x_{1}x_{2}^{2}-x_{1}x_{2}x_{3}^{2}+x_{1}x_{2}x_{3}-x_{2}^{2}x_{3}-x_{2}x_{3}^{2}+x_{2}x_{3}\\
+x_{3}^{2}\,:\,x_{1}x_{2}+x_{1}x_{3}-x_{1}+x_{2}^{2}+x_{2}x_{3}-2x_{2}-x_{3}+1\,:\,x_{2}^{2}+x_{2}x_{3}-x_{2}-x_{3}),\\
(-x_{1}x_{2}^{2}-x_{1}x_{2}x_{3}+x_{1}x_{2}-x_{2}x_{3}+x_{3}\,:\,0\,:\,x_{2}^{2}+x_{2}x_{3}-x_{2}-x_{3}),\\
(-x_{2}^{2}-x_{2}x_{3}+x_{2}+x_{3}\,:\,x_{1}x_{2}+x_{1}x_{3}-x_{1}-x_{2}x_{3}-x_{3}^{2}+x_{3}\,:\,x_{2}^{2}+x_{2}x_{3}\\
-x_{2}-x_{3}),(0:1:1).
\end{array}\label{eq:L23(L7)=00003DM7}
\end{equation}
}
A computation in \texttt{OSCAR} yields the following concrete description of the moduli space $\cR_{7}=\cR(M_7)$.
\begin{prop}
The moduli space $\cR_{7}$ is an open sub-scheme of $Z_7$: for $x\in\cR_{7}$,
the line arrangement $\cA=\cC_{0}(x)\cup\cC_{1}(x)$ is a realization
of $M_{7}$, and conversely any realization of $M_{7}$ is projectively
equivalent to a unique such line arrangement. 
\end{prop}

The complement of $\cR_{7}$ in $Z_7$ is the union of $20$ irreducible
curves described in Section \ref{subsec:Open-surface-R7}.  

From the definition of the matroid $M_{7},$ if $\cA=\cC_{0}\cup\cC_{1}$
is a realization of $M_{7}$, one has $\ldt(\cC_{0})=\cC_{1}$, but
the following result on $\cC_{2}=\ldt(\cC_{1})$ is unexpected:
\begin{thm}
\label{thm:L23-rational-self-map}Let $\cA_{0}=\cC_{0}\cup\cC_{1}$
be a generic realization of $M_{7}$ and define $\cC_{2}=\ldt(\cC_{1})$.
The labeled line arrangement $\cA_{1}=\cC_{1}\cup\cC_{2}$ is again
a realization of $M_{7}$. The operator $\ldt$ induces a rational
self-map on the schemes $\ksU$ of all realizations of $M_{7}$ and
its moduli space $\cR_{7}$. 
\end{thm}

We denote by $\lldt:Z_7\dashrightarrow Z_7$ the rational self-map
on $Z_7$ induced by $\ldt$.   
\begin{proof}
Up to projective automorphism, one can suppose that the line arrangement
$\cA_{0}$ is of the form $\cA_{0}=\cC_{0}(x)\cup\cC_{1}(x)$ for
$x$ generic in $Z_7$: concretely, we use $x=(x_{1},x_{2},x_{3})$, where $x_{1},x_{2},x_{3}\in\CC(Z_7)$
are considered as rational functions.  
A direct computation (with Magma) then shows
that $\cC_{2}=\ldt(\cC_{1})$ is a line arrangement of seven lines. It has a
canonical labeling as described in the previous Subsection 
and we then check that the matroid associated to
$\cC_{1}\cup\cC_{2}$ is equal to $M_{7}$, so that $\cC_{1}\cup\cC_{2}$
is a realization of $M_{7}$. Using the period map, one computes 
$\lldt$ and obtain that it is a dominant rational map. The reader can find the polynomials defining 
$\lldt$ in an ancillary file of this paper on arXiv; it can be also retrieved from 
the polynomials given in Section \ref{subsec:Practical-SemiConj}.
That describes action
of $\ldt$ on the space of realization $\ksU$ and on the moduli space
$\cR_{7}$.
\end{proof}

\subsection{\label{subsec:Open-surface-R7}The open surface $\protect\cR_{7}$
inside $Z_7$}

The scheme $Z_7\setminus\cR_{7}$ is the union of the following curves:

\begin{itemize}
\item The $12$ lines 
\begin{alignat*}{9}
&L_{1}: y_{2}=y_{3}=0, 								&&L_{2}:y_{1}=y_{3}=0,							&&L_{3}:y_{2}=y_{4}=0,\\
&L_{4}:y_{1}-y_{3}=y_{4}=0,						&&L_{5}: y_{1}=y_{4}=0,							&&L_{6}:y_{2}-y_{4}=y_{3}=0,\\
&L_{7}:y_{1}-y_{3}-y_{4}=y_{2}+y_{3}=0,\,	&&L_{8}: y_{1}-y_{3}=y_{2}+y_{3}=0,	&&L_{9}:y_{2}+y_{3}=y_{4}=0,\\
&L_{10}:y_{1}-y_{4}=y_{3}=0,					&&L_{11}: y_{1}-y_{3}=y_{2}-y_{4}=0,\,	&&L_{12}:y_{1}=y_{2}-y_{4}=0.
\end{alignat*}
These lines are also the lines contained in the quartic surface $Z_7$
that contain at least two double points of $Z_7$.

\item  The conic $C_{o}$ defined by $y_{1}y_{3}-y_{3}^{2}-y_{1}y_{4}=y_{2}+y_{3}-y_{4}=0$. 

\item  Seven curves $E_{1},\dots,E_{7}$ of geometric genus one.
For example, one of these curves is given by
\[
y_{1}^{2}-2y_{1}y_{3}+y_{3}^{2}-y_{1}y_{4}=y_{2}^{2}+y_{2}y_{3}+y_{1}y_{4}-y_{3}y_{4}-y_{4}^{2}=0.
\]
\end{itemize}
The $j$-invariant of the normalizations of the curves $E_{i}$ is equal to 
 $-5^{6}/28$. The elliptic curve with this $j$-invariant is known
as the modular curve $X_{1}(14)$ parametrizing pairs $(E,t)$ where
$E$ is an elliptic curve and $t$ is an order $14$ torsion element
of $E$. For a generic point $p$ on the curves $E_{1},\dots,E_{7}$,
the line arrangement $\cC_{0}(p)$ with normal vectors as in \eqref{eq:of-mathcal-Co}
is well-defined. The line arrangement $\cC_{1}=\ldt(\cC_{0})$ has
seven lines, but its singularities are $t_{2}=6,t_{3}=5$, and one has
$\ldt(\cC_{1})=\emptyset$.
Moreover, the singularities of $\cC_{0}\cup\cC_{1}$
are $t_{2}=13,t_{3}=26$.

The image of the curves $C_{o},E_{1},\dots,E_{7}$ under the map $\lldt$ are
lines $L_{k}$; when defined, the image of the lines $L_{k}$ are
lines $L_{k'}$ or points.

\subsection{\label{subsec:Degree-L23}The degree of $\protect\lldt$}

Recall that $\lldt:Z\dashrightarrow Z$ denotes the action of the
operator $\ldt$ on the K3 surface $Z_7$. One has:
\begin{thm}
\label{thm:action-L23-auto} The operator $\lldt$ acts on $Z_7$ as
a degree $4$ rational self-map.
\end{thm}

In order to prove \Cref{thm:action-L23-auto}, let us describe the
period map: \\
Let $\ell_{1},...,\ell_{7}$ be the lines of $\cC_{0}$ with normal
vectors as in Equation \eqref{eq:of-mathcal-Co}. Let us denote by $p_{i,j}$
the intersection point of the lines $\ell_{i}$ and $\ell_{j}$. The
point $p_{5,7}$ is $(1:x_{2}:x_{3})$, so that one may recover $x_{2},x_{3}$
from the knowledge of that point. Also the point $p_{1,7}$ is 
\begin{equation}
(0:-x_{1}x_{2}{}^{2}-x_{1}x_{2}x_{3}+x_{1}x_{2}+x_{2}{}^{2}-x_{2}:x_{2}x_{3}+x_{3}{}^{2}-x_{3}),\label{eq:normalLine17}
\end{equation}
this is linear in $x_{1}$, so that from the knowledge of $p_{5,7}$
and $p_{1,7}$, one may recover the point $(x_{1},x_{2},x_{3})\in Z_7$. 
\begin{proof}[Proof of \Cref{thm:action-L23-auto}]
Let $A\in PGL_{3}(\CC)$ be the projective transformation that sends
the first four lines of $\cC_{1}$ to the four lines having the 
same normal vectors as the one of $\cC_{0}$. Let $\cC_{1}'=(\ell_{1}',\dots,\ell_{7}')$
be the image of $\cC_{1}$ by $A$. Using the period map, one can determine
the points $p_{5,7}'$ and $p_{1,7}'$ and we obtain a point $x'=(x_{1}',x_{2}',x_{3}')$
(in the function field of $Z_7$). The line arrangements $\cC_{0}(x_{1}',x_{2}',x_{3}')$
and $\cC_{1}'$ are equal, and the action of $\ldt$ on $Z_7$ is through
the map 
\[
\lldt:(x_{1},x_{2},x_{3})\to(x_{1}',x_{2}',x_{3}').
\]
The rational self-map $\lldt:Z_7\dashrightarrow Z_7$ is studied in Section
\ref{subsec:Practical-SemiConj}. 

Let us compute the degree of $\lldt$; we apply the method from \cite{Voisin}.
Let $f(x_{1},x_{2},x_{3})$ be the equation of the quartic $Z_7$ in the
chart $U_{4}:y_{4}\neq0$. The space of global non-vanishing differential
$2$-forms is generated by a form $\o$, which one can choose so that
on an open set of $U_{4}$ one has: 
\[
\omega=\frac{dx_{2}\wedge dx_{3}}{\partial f/\partial x_{1}}.
\]
The rational self-map $\lldt$ preserves $U_{4}$, and by a direct
computation one obtains that 
\[
\lldt^{*}\o=-2\o.
\]
 The above expression shows that when applying $\lldt$, the volume
form $\om\bar{\om}$ is multiplied by $4$, which gives the degree
of $\lldt$.
\end{proof}

\subsection{\label{subsec:The-automorphism-group}Action of $\protect\aut(M_{7})$
on the K3 surface $Z_7^{s}$}

The automorphism group of $M_{7}$ is generated by the order $7$
and $6$ permutations {\small{}
\[
\s_{1}=(1,7,4,3,6,5,2)(8,14,11,10,13,12,9)\text{ and }\s_{2}=(1,3,5,6,7,2)(8,10,12,13,14,9).
\]
}with $1'=8,\dots,7'=14$. 
This group is isomorphic to the Frobenius group $F_{7}=\ZZ/6\ZZ\rtimes\ZZ/7\ZZ$.
These automorphisms act on the K3 surface $Z$. 
\begin{prop}
The action of $\aut(M_{7})$ on $Z_7$ is faithful.
\end{prop}
\begin{proof}
As in the proof of Theorem \ref{thm:action-L23-auto}, let $\cC_{0}=\cC_{0}(x_{1},x_{2},x_{3})$
be the generic line arrangement in $\cR_{7}$, where $x_{1},x_{2},x_{3}\in\CC(Z_7)$
are considered as rational functions. 

For $\s\in\aut(M_{7}),$ let $\cC_{0}^{\s}$ be the image of $\cC_{0}$
under the action of $\s$ (that is just the permutation of the lines
under $\s$). We apply the period map to the line arrangement $\cC_{0}^{\s}$,
where $\cC_{0}=\cC_{0}(x)$. Using the period map, we obtain the point
$\s(x)=(x_{1}',x_{2}',x_{3}')$ which is a zero of the equation of
$Z_7$ and such that $\cC_{0}(\s(x))$ is projectively equivalent to
$\cC_{0}^{\s}$.

When $\s=\s_{1}$, the automorphism $\s_{1}$ acts on $Z_7$ through
the map in $\PP^{3}$ given by the ring homomorphism which to $(y_{1},y_{2},y_{3},y_{4})$
associates{\small{}
\[
\begin{array}{c}
(y_{1}y_{2}^{2}y_{3}+y_{1}y_{2}y_{3}^{2}-y_{2}^{2}y_{3}^{2}-y_{2}y_{3}^{3}-y_{1}y_{2}y_{3}y_{4}-y_{2}^{2}y_{3}y_{4}+y_{2}y_{3}y_{4}^{2}+y_{3}^{2}y_{4}^{2},\\
y_{1}y_{2}^{3}+y_{1}y_{2}^{2}y_{3}+y_{2}^{2}y_{3}^{2}+y_{2}y_{3}^{3}-2y_{1}y_{2}^{2}y_{4}-y_{1}y_{2}y_{3}y_{4}-2y_{2}y_{3}^{2}y_{4}-y_{3}^{3}y_{4}+y_{1}y_{2}y_{4}^{2}+y_{3}^{2}y_{4}^{2},\\
y_{1}y_{2}^{2}y_{3}+y_{1}y_{2}y_{3}^{2}-y_{2}^{2}y_{3}^{2}-y_{2}y_{3}^{3}-y_{1}y_{2}y_{3}y_{4}+y_{2}y_{3}^{2}y_{4},\\
y_{2}^{3}y_{3}+2y_{2}^{2}y_{3}^{2}+y_{2}y_{3}^{3}-2y_{2}^{2}y_{3}y_{4}-3y_{2}y_{3}^{2}y_{4}-y_{3}^{3}y_{4}+y_{2}y_{3}y_{4}^{2}+y_{3}^{2}y_{4}^{2}).
\end{array}
\]
}For $\s_{2}$, we obtain that it acts on the surface $Z_7$ through
the map which to $(y_{1},y_{2},y_{3},y_{4})$ associates 
\[
\begin{array}{c}
(-y_{2}^{2}y_{3}-y_{2}y_{3}^{2}+y_{2}y_{3}y_{4},-y_{1}y_{2}y_{3}+y_{2}y_{3}^{2}+y_{2}y_{3}y_{4}-y_{3}y_{4}^{2},\\
y_{1}y_{2}^{2}+y_{1}y_{2}y_{3}-y_{2}^{2}y_{3}-y_{2}y_{3}^{2}-y_{1}y_{2}y_{4}+y_{2}y_{3}y_{4},y_{2}y_{3}y_{4}-y_{3}y_{4}^{2});
\end{array}
\]
thas map is a birational transformation of $\PP^{3}$.
In order
to check that the action of $\aut(M_{7})$ is faithful on $Z_7$, it
is then enough to check that the orbit of one point (for example the
point $(-6:-25/8:5:1)$ in $Z_7$) has $42$ elements, which is a direct
computation.

The fixed points under the order seven element $\s_{1}$ are the singularities
$s_{5},s_{7},s_{8}$; (there is a unique conjugacy class of elements
of order $7$ in $F_{7}$). 

The fixed points locus of $\s_{2}$ and the order $3$ automorphism
$\s_{2}^{2}$ acting on $Z_7$ are: 
\begin{enumerate}[(i)]
\item The $A_{3}$ singularity $(0:1:0:1)$,
\item The four points $p=(r^{2}+1:r{}^{2}-r+2:r:1)$ where $r$ is any
complex root of $X^{4}-X^{3}+3X^{2}-X+1$. These points are in $Z_7\setminus\cR_{7}$;
they are periodic of period $2$ for the rational self-map $\lambda_{\{2\},\{3\}}$,
moreover the (unlabeled) line arrangements $\cC_{0}(p),\cC_{1}=\ldt(\cC_{0}),\cC_{2}=\ldt(\cC_{1})$
have $7,10$ and $37$ lines respectively. It seems likely that the
number of lines of the sequence $\cC_{n+1}=\ldt(\cC_{n})$ goes to
infinity.
\item The points $(w+1:-w:w:1)$ where $w^{2}+w+1=0$, which are fixed
by the rational self-map $\lambda_{\{2\},\{3\}}$; these two points are
in $Z_7\setminus\cR_{7}$. 
\end{enumerate}
The fixed-point locus of the involution $\s_{2}^{3}$ acting on $Z_7$
is the union of the line $L_{7}$ and a curve $E_{j}$, which is in
$Z_7\setminus\cR_{7}$ (see Section \ref{subsec:Open-surface-R7}).
There is a unique conjugacy class of involutions in $F_{7}$, so that
similarly, any involution from $\aut(M_{7})$ fixes a curve and a
line.
\end{proof}

There is an open set in the quotient surface $Z_7/\aut(M_{7})$ which
parametrizes unlabeled line arrangements $\cC_{0}^{o}$ associated
to $\cC_{0}$ in $\cR_{7}$. One has:
\begin{cor}
The surface $Z_7/\aut(M_{7})$ is rational.
\end{cor}
\begin{proof}
Since an involution of $\aut(M_{7})$ fixes a one dimensional curve,
it is non-symplectic (see \cite{Huybrechts}), thus $Z_7/\aut(M_{7})$ is rational. 
\end{proof}
For a labeled line arrangement $\cC_{0}=(\ell_{1},\dots,\ell_{k})\in\ksU$
and $j\in\{1,\dots,7\}$, let us denote by $H_{j}(\cC_{0})$ the line
arrangement $H_{j}=\sum_{k\neq j}\ell_{k}$. The labeled line arrangement
$\cC_{1}=\ldt(\cC_{0})$ is 
\[
\cC_{1}=(\ldt(H_{1}),\dots,\ldt(H_{7})).
\]
An element $\s\in\aut(M_{7})$ permutes the lines of $\cC_{0}$: it
can also be seen as a permutation of $\{1,\dots,7\}$. The $\s(j)^{th}$
line of $\s.\cC_{1}$ is $\ldt(H_{\s(j)}(\cC_{0}))$. Since
\[
H_{\s(j)}(\cC_{0})=\sum_{k\neq j}\ell_{\s(k)}=H_{j}(\s.\cC_{0}),
\]
the $\s(j)^{th}$ line of $\s.\cC_{1}$ is $\ldt(H_{j}(\s.\cC_{0}))$.
Thus 
\[
\s.\ldt(\cC_{0})=(\ldt(H_{\s(1)}),\dots,\ldt(H_{\s(7)}))=\ldt(\s.\cC_{0})
\]
and we obtain that:
\begin{prop}
The action of $\aut(M_{7})$ commutes with the action of $\ldt$, that is for all $\s\in\aut(M_{7})$ it holds that
\[
\ldt\circ\s=\s\circ\ldt.
\]
\end{prop}

\begin{rem}
The group $\aut(M_{7})$ acts faithfully on the surface $Z_7\subset\PP^{3}$,
but does not extend canonically to a well-defined action on the ambient
space $\PP^{3}$. For example, the action of $\s_{2}$ we computed
is the restriction of an order $6$ birational map $\tilde{\s_{2}}$
of $\PP^{3}$; in particular $\tilde{\s_{2}}^{3}$ is a birational
involution of $\PP^{3}$ defined by degree $5$ coprime polynomials.
If instead one starts with $\s_{2}^{3}=(1,6)(2,5)(3,7)(8,13)(9,12)(10,14)$
and computes the action of $\widetilde{\s_{2}^{3}}$ on $Z_7$ as we did
above for $\s_{1}$ and $\s_{2}$, one obtains that, surprisingly,
the defining coprime polynomials of the rational map $\widetilde{\s_{2}^{3}}:\PP^{3}\dashrightarrow\PP^{3}$
have degree $4$, although the maps $\widetilde{\s_{2}^{3}}$ and
$\tilde{\s_{2}}^{3}$ have the same effect on $Z_7$. Moreover although
$(\widetilde{\s_{2}^{3}})^{2}$ is the identity on the surface $Z_7$,
it is not the identity on $\PP^{3}$ (it is defined by degree $6$
coprime polynomials). Moreover, one can compute that the rational map
$(\widetilde{\s_{2}^{3}})^{4}$ is defined by degree $21$ coprime
polynomials.
\end{rem}

\subsection{\label{subsec:Fibration-preserved-by}Fibration preserved by $\protect\lldt$ and the elliptic modular surface $\Xi_{1}(7)$}

The line $L_{6}:\,y_{2}-y_{4}=y_{3}=0$ is contained in the surface
$Z_7$. Let $\g:Z_{7}\to\PP^{1}$ be the elliptic fibration induced
by the projection from that line. One obtains a smooth cubic affine
model $A$ in $\AA^{2}_{\QQ(t)}=\AA^{2}_{\QQ(t)}(x,y)$ of that elliptic fibration by
substituting $(x,1+ty,y,1)$ in the equation of $Z_{7}$. A computation
shows that $Z_{7}^{s}\to\PP^{1}$ is (isomorphic to) the elliptic
surface $Y$ associated to the elliptic curve $E_{/\QQ(t)}$ with
Weierstrass model
\[
E:\,y^{2}=x^{3}+\tfrac{(t^{4}-2t^{3}+3t^{2}+6t+1)}{(t+1)^{2}}x^{2}+\tfrac{8t^{3}(t^{2}-t-1)}{(t+1)^{3}}x+16\tfrac{t^{6}}{(t+1)^{4}}.
\]
The map between $A$ and $Y$ sends $(0,0)$ to the zero section.
The elliptic fibration $Y\to\PP^{1}$ has singular fibers $3I_{7}+3I_{1}$
at the points 
\[
\infty,0,-1,t^{3}-5t^{2}-8t-1=0,
\]
respectively. 

We recall that the curve $X_{1}(7)$ parametrizes (up to isomorphisms)
the pairs $(E,p)$ where $E$ is an elliptic curve and $p$ is a torsion point
of order $7$ on $E$. A Weierstrass
model $E'$ of the elliptic modular surface $\Xi_{1}(7)$ over the
curve $X_{1}(7)\simeq\PP^{1}$ is computed in \cite{TY}. The $j$-invariant maps $j_{E}(t),j_{E'}(t)\in\QQ$
of $E$ and $E'$ are related by the equality $j_{E'}(t)=j_{E}(-\tfrac{1}{t})$,
which shows that $E$ is isomorphic to $E'$ and $Z_7^{s}$ is isomorphic
to the elliptic modular surface $X_{1}(7)$. 

The Mordell--Weil group of $E$ is isomorphic to $\ZZ/7\ZZ$; it is
generated by the point 
\[
p_{t}=(0:4t^{3}:(t+1)^{2})\in E.
\]
 We thus obtained the first part of the following theorem:
\begin{thm}
	\begin{enumerate}[a)]
		\item The K3 surface $Z_7^{s}$ is isomorphic to the modular elliptic surface
		$\Xi_{1}(7)$.
		\item The rational self-map $\lldt$ preserves the elliptic fibration
		$\g:Z_{7}\to\PP^{1}$ and acts on the base curve $\PP^{1}$ through
		the order $3$ map $t\to-1/(t+1)$. There exists an automorphism $\s_{0}$
		coming from $\aut(M_{7})$ such that $\s_{0}\lldt$ preserves the
		fibration $\g$ and acts on $E$ as the multiplication by $2$ map.
	\end{enumerate}
\end{thm}

The last property implies that the operator $\ldt$ preserves the
moduli interpretation of $X_{1}(7)$.
\begin{proof}
Using the period map and the function field of $A$, one computes
that the action of $\lldt$ on the base $\PP^{1}$ of the fibration
$A\to\PP^{1}$ is through the map $t\to-1/(t+1)$.

An automorphism $\s\in\aut(M_{7})$ acts on the surface $Z_7\cap\{y_{4}\neq0\}$ and 
on the affine model $A$. Using the period map
and again the generic point of $A$, one computes the action of the
rational self-maps $\s\lldt$ ($\s\in\aut(M_{7})$) on $A$. For
$14$ of these maps, the action on the base curve $\PP^{1}$ is trivial.
This is the case for example for 
\[
\s_{0}=(1,2,4)(3,6,7)(8,9,11)(10,13,14).
\]
The map $\mu=\s_{0}\lldt$ also acts on $E$, one can thus compute its action
on the generic point of $E$. Knowing that action, we are now able
to compute the pull-back of a non-zero holomorphic one-form $\omega$
by $\mu$, which is: $\mu^{*}\o=2\o$. Using the seven torsion points,
one computes that $\mu$ fixes the origin, thus $\mu=[2]$. 
\end{proof}

Among the $12$ lines contained in $Z_7$, in the complement of $\cR_{7}$,
eight are contained in the singular fibers of the fibration $\g$,
and $4$ are sections.

Using the pull-back to $Z_7^{s}$ of the lines contained in $Z_7$ and
the $\cu$-curves of the desingularization, one may compute the N\'eron--Severi
lattice of $Z_7^{s}$, and obtain that it has discriminant $-7$ and
rank $20$. The modular elliptic surface $\Xi_{1}(7)$ is well-known
and studied; it is known as the unique K3 surface with N\'eron--Severi
lattice of rank $20$ and discriminant $-7$: we obtain
that way another proof that $Z^s$ is isomorphic to $\Xi_1(7)$. The inequivalent fibrations
of $\Xi_{1}(7)$ have been classified (see \cite{Lecacheux}). Another
remarkable fact is that $\Xi_{1}(7)$ is a ball-quotient surface:
there exists a co-compact lattice $\G$ in the automorphism group
of the unit ball $\BB_{2}$ such that $\Xi_{1}(7)\simeq\BB_{2}/\G$~\cite{Naruki}.
The automorphism group of $\Xi_{1}(7)$ is studied
in \cite{Ujikawa}.

\subsection{\label{subsec:Practical-SemiConj}The K3 surface $Z_7^{s}$ is semi-conjugated
to the plane}

The rational self-map $\l_{\{2\},\{3\}}$ acting on the quartic $Z_7\hookrightarrow\PP^{3}(y_{1},\dots,y_{4})$
is defined by 
\[
\lldt=(P_{1}:\dots:P_{4}),
\]
where $P_{1},\dots,P_{4}$ are four homogeneous degree $11$ polynomials
computed via the period map. These polynomials  are given in the 
ancillary file in the arXiv version of this paper; they also
may be obtained from the polynomials $Q_{1},Q_{2},Q_{3}$ and
$R$ below. A remarkable fact about the polynomials $P_{1},\dots,P_{4}$
is that 
\begin{equation}
\deg_{y_{1}}(P_{1})=1,\,\deg_{y_{1}}(P_{2})=\deg_{y_{1}}(P_{3})=\deg_{y_{1}}(P_{4})=0,\label{eq:Deg_y1}
\end{equation}
where $\deg_{y_{1}}$ denote the degree relative to the variable $y_{1}$. 

Let us define the polynomials $\tilde{P_{k}}=P_{k+1}(0,z_{1},z_{2},z_{3})$
for $k\in\{1,2,3\}$ (where $z_{1},z_{2},z_{3}$ are the three coordinates
on the plane $\PP^{2}:y_{1}=0$). The polynomials $\tilde{P}_{k}$, $k\in\{1,2,3\}$
define a rational self-map $F:\PP^{2}\dashrightarrow\PP^{2}$; the
base locus of the linear system generated by $\tilde{P_{1}},\tilde{P_{2}},\tilde{P_{3}}$
is the quintic curve $B$ defined by
\[
\begin{array}{c}
Q=z_{1}^{3}z_{2}^{2}+2z_{1}^{2}z_{2}^{3}+z_{1}z_{2}^{4}+2z_{1}^{3}z_{2}z_{3}+4z_{1}^{2}z_{2}^{2}z_{3}+2z_{1}z_{2}^{3}z_{3}+z_{1}^{3}z_{3}^{2}\\
-4z_{1}^{2}z_{2}z_{3}^{2}-9z_{1}z_{2}^{2}z_{3}^{2}-4z_{2}^{3}z_{3}^{2}-2z_{1}^{2}z_{3}^{3}+2z_{1}z_{2}z_{3}^{3}+4z_{2}^{2}z_{3}^{3}+z_{1}z_{3}^{4}.
\end{array}
\]
That curve is irreducible, has geometric genus $1$ and its normalization
has $j$-invariant $-5^{6}/28$. By removing the base locus $B$, one
obtains that the rational self-map $F$ is defined by the following
degree $6$ polynomials 

\begin{align*}
Q_{1}=&z_{1}Q,\\
Q_{2}=&-z_{1}^{5}z_{2}-3z_{1}^{4}z_{2}^{2}-3z_{1}^{3}z_{2}^{3}-z_{1}^{2}z_{2}^{4}+z_{1}^{4}z_{2}z_{3}+2z_{1}^{3}z_{2}^{2}z_{3}+z_{1}^{2}z_{2}^{3}z_{3}\\
&+z_{1}^{3}z_{2}z_{3}^{2}+2z_{1}^{2}z_{2}^{2}z_{3}^{2}+z_{1}z_{2}^{3}z_{3}^{2}-z_{1}^{2}z_{2}z_{3}^{3}+z_{2}^{3}z_{3}^{3}-z_{2}^{2}z_{3}^{4},\\
Q_{3}=&2z_{1}^{4}z_{2}z_{3}+4z_{1}^{3}z_{2}^{2}z_{3}+2z_{1}^{2}z_{2}^{3}z_{3}+z_{1}^{4}z_{3}^{2}-4z_{1}^{3}z_{2}z_{3}^{2}-8z_{1}^{2}z_{2}^{2}z_{3}^{2}-3z_{1}z_{2}^{3}z_{3}^{2}\\
&-2z_{1}^{3}z_{3}^{3}+2z_{1}^{2}z_{2}z_{3}^{3}+4z_{1}z_{2}^{2}z_{3}^{3}+z_{2}^{3}z_{3}^{3}+z_{1}^{2}z_{3}^{4},
\end{align*}

and the indeterminacy locus of $F=(Q_{1}:Q_{2}:Q_{3})$ are the $8$
points {\small{}
\[
q_{1}=(0:0:1),q_{2}=(1:0:1),q_{3}=(0:1:0),q_{4}=(-1:1:0),q_{5}=(1:0:0),\,q_{r}=(-r{}^{2}+2r:r:1)
\]
}where $r$ is any root of $X^{3}-4X^{2}+3X+1$ (the field $\QQ(r)$
is the degree $3$ real subfield of $\QQ(\zeta_{7})$). 

Let us define the projection map $\pi_{1}:Z_7\to\PP^{2}(z_{1},z_{2},z_{3})$
from the point $s_{8}:y_{2}=y_{3}=y_{4}=0$ contained in surface $Z_7$.
This point is an $A_{3}$ singularity on $Z_7$, in particular it has multiplicity
$2$, thus the map $\pi_{1}$ from the quartic to the plane has degree
$2$. One has: 
\begin{lem}
The branch loci of $\pi_{1}$ is the union of the quintic curve $B=\{Q=0\}$
and the line $L:z_{1}=0$. 
\end{lem}

\begin{proof}
The ramification locus of $\pi_{1}$ is the discriminant of the equation
of $Z_7$ (given in \eqref{eq:quarti-Z}) with respect to the variable
$y_{1}$. The image of the ramification curve by $\pi_{1}$ is the
curve $B+L$.
\end{proof}
The curve $B$ has singularities of type $A_{4},A_{4},A_{2}$ at the points
$q_{2},q_{4},q_{5}$, respectively. The union $L+B$ has singularities
of type $A_{1},A_{3},A_{4},A_{3},A_{4},A_{2}$ at the points
$q_{0}=(0:1:1),q_{1},q_{2},q_{3},q_{4},q_{5}$, respectively. 

A direct computation shows that 
\[
Q_{1}(Q_{1},Q_{2},Q_{3})=Q_{1}\,R^{2}
\]
for $R=\tfrac{1}{8}z_{2}^{2}(z_{1}-z_{3})^{2}R_{4}R_{7}$, where {\small{}
\begin{align*}
R_{4}=&z_{1}^{4}+2z_{1}^{3}z_{2}+z_{1}^{2}z_{2}^{2}-z_{1}^{2}z_{3}^{2}-z_{1}z_{2}z_{3}^{2}-z_{2}z_{3}^{3},\\
R_{7}=&z_{1}^{6}z_{2}+4z_{1}^{5}z_{2}^{2}+6z_{1}^{4}z_{2}^{3}+4z_{1}^{3}z_{2}^{4}+z_{1}^{2}z_{2}^{5}+z_{1}^{6}z_{3}-7z_{1}^{4}z_{2}^{2}z_{3}-11z_{1}^{3}z_{2}^{3}z_{3}-6z_{1}^{2}z_{2}^{4}z_{3}-z_{1}z_{2}^{5}z_{3}\\
&-z_{1}^{5}z_{3}^{2}+3z_{1}^{3}z_{2}^{2}z_{3}^{2}+2z_{1}^{2}z_{2}^{3}z_{3}^{2}+3z_{1}^{2}z_{2}^{2}z_{3}^{3}+5z_{1}z_{2}^{3}z_{3}^{3}+2z_{2}^{4}z_{3}^{3}-2z_{1}z_{2}^{2}z_{3}^{4}-2z_{2}^{3}z_{3}^{4}-z_{1}z_{2}z_{3}^{5}.
\end{align*}
}The images of the curves $z_{2}=0,\,z_{1}-z_{3}=0$ and $R_{4}=0$
by the rational self-map $F=(Q_{1}:Q_{2}:Q_{3})$ of $\PP^{2}$ are
the indeterminacy points $q_{2}, q_{4}, q_{2}$, respectively. The image
of the curve $R_{7}=0$ under the map $F$ is the quintic curve $B$. The image
of the quintic curve $B$ under the map $F$ is the line $L:z_{1}=0$. The rational
map $F$ preserves $L$ and the action of $F$ on $L$ is through the
map $(z_{2}:z_{3})\to(z_{2}-z_{3}:z_{3})$.

From the above description and Subsection \ref{subsecTheoretical-SConj}, the surface $Z_7^{s}$ is the minimal desingularization
of the double cover 
\[
X:\{y^{2}=Q_{1}(z_{1},z_{2},z_{3}\}\hookrightarrow\PP(3,1,1,1)
\]
branched over $L+B$. The birational map between $X$ and $Z_7$ is
given by the equalities $y_{i+1}=z_{i}$ for $i\in\{1,2,3\}$ and
\[
y_{1}=\tfrac{1}{2}(y+z_{2}^{2}z_{3}+z_{2}z_{3}^{2}+z_{2}^{2}z_{4}-z_{2}z_{3}z_{4}-z_{2}z_{4}^{2})/(z_{2}^{2}+z_{2}z_{3}-z_{2}z_{4}).
\]
We continue to denote by $\lldt$ the rational self-map 
\[
(y;z)\to(yR(z);F(z)).
\]
 Applying the results of \Cref{subsecTheoretical-SConj},
we obtain that:
\begin{thm}
The dynamical system $(Z_7^{s},\lldt)$ is semi-conjugated to $(\PP^{2},F)$.
\end{thm}

\begin{rem}
a) The degrees of the coprime polynomials defining the rational maps
$F,F^{2},F^{3}$ are $6,21,82$, respectively. //
b) It would be interesting to construct rational self-maps on some 
other degree two K3 surfaces. 
\end{rem}

\section{\label{sec:The-Octagon-and}The Octagon and the operator $\protect\L_{\{2\},\{3,4\}}$}

\subsection{The matroid $M_{8}$ constructed from the regular octagon}

Consider the $16$ lines in Figure \ref{fig:Octagon}: the black
lines $\ell_{1},\dots,\ell_{8}$ are the $8$ lines of the regular
octagon $\cC_{1}$ and the blue lines $\ell_{1}',\dots,\ell_{8}'$
are the 8 lines symmetries of $\cC_{0}$. The image of $\cC_{0}=\ell_{1}+\dots+\ell_{8}$
by the operator $\L_{\{2\},\{3,4\}}$ is the line arrangement $\cC_{1}=\ell_{1}'+\dots+\ell_{8}'$. 

\begin{figure}[h]
\begin{center}

\begin{tikzpicture}[scale=0.5]

\clip(-11.367310671897055,-7.696664753899389) rectangle (8.567727627955502,7.634088050301914);
\draw [line width=0.25mm,domain=-11.367310671897055:8.567727627955502] plot(\x,{(-4.84212581620924-1.4283556979968257*\x)/-1.4283556979968262});
\draw [line width=0.25mm] (-2.375,-7.696664753899389) -- (-2.375,7.634088050301914);
\draw [line width=0.25mm,domain=-11.367310671897055:8.567727627955502] plot(\x,{(-4.8278422592292705-1.4283556979968255*\x)/1.4283556979968255});
\draw [line width=0.25mm,domain=-11.367310671897055:8.567727627955502] plot(\x,{(-4.915378509953586-0.*\x)/2.02});
\draw [line width=0.25mm,domain=-11.367310671897055:8.567727627955502] plot(\x,{(-5.008831203697934--1.4283556979968255*\x)/1.428355697996826});
\draw [line width=0.25mm] (2.501711395993651,-7.696664753899389) -- (2.501711395993651,7.634088050301914);
\draw [line width=0.25mm,domain=-11.367310671897055:8.567727627955502] plot(\x,{(-5.023114760677902--1.428355697996826*\x)/-1.4283556979968255});
\draw [line width=0.25mm,domain=-11.367310671897055:8.567727627955502] plot(\x,{(-4.935578509953586-0.*\x)/-2.02});
\draw [line width=0.25mm,color=blue,domain=-11.367310671897055:8.567727627955502] plot(\x,{(--0.034483556979966495-0.*\x)/6.89671139599365});
\draw [line width=0.25mm,color=blue,domain=-11.367310671897055:8.567727627955502] plot(\x,{(-0.10359495297362065--2.02*\x)/4.8767113959936506});
\draw [line width=0.25mm,color=blue,domain=-11.367310671897055:8.567727627955502] plot(\x,{(-0.2845838974422854--4.8767113959936506*\x)/4.876711395993652});
\draw [line width=0.25mm,color=blue,domain=-11.367310671897055:8.567727627955502] plot(\x,{(-0.29886745442225093--4.8767113959936506*\x)/2.02});
\draw [line width=0.25mm,color=blue] (0.06335569799682518,-7.696664753899389) -- (0.06335569799682518,7.634088050301914);
\draw [line width=0.25mm,color=blue,domain=-11.367310671897055:8.567727627955502] plot(\x,{(-0.3190674544222505--4.8767113959936506*\x)/-2.02});
\draw [line width=0.25mm,color=blue,domain=-11.367310671897055:8.567727627955502] plot(\x,{(-0.33335101140221646--4.8767113959936506*\x)/-4.8767113959936506});
\draw [line width=0.25mm,color=blue,domain=-11.367310671897055:8.567727627955502] plot(\x,{(-0.15236206693355525--2.02*\x)/-4.8767113959936506});

\begin{scriptsize}
\draw [fill=black] (-2.375,1.015) circle (0.99mm);
\draw[color=black] (-2.8,1.5) node {$p_{1,2}$};

\draw [fill=black] (-2.375,-1.005) circle (0.99mm);
\draw[color=black] (-3.1,-0.9) node {$p_{2,3}$};

\draw [fill=black] (-0.9466443020031745,-2.4333556979968254) circle (0.99mm);
\draw[color=black] (-0.9064752188652873,-3) node {$p_{3,4}$};

\draw [fill=black] (1.0733556979968246,-2.4333556979968254) circle (0.99mm);
\draw[color=black] (1.8,-2.8) node {$p_{4,5}$};

\draw [fill=black] (2.5017113959936506,-1.005) circle (0.99mm);
\draw[color=black] (3.3,-0.9) node {$p_{5,6}$};

\draw [fill=black] (2.501711395993651,1.015) circle (0.99mm);
\draw[color=black] (3.3,1.0) node {$p_{6,7}$};

\draw [fill=black] (1.0733556979968255,2.443355697996825) circle (0.99mm);
\draw[color=black] (1.5,2.8) node {$p_{7,8}$};

\draw [fill=black] (-0.9466443020031738,2.4433556979968256) circle (0.99mm);
\draw[color=black] (-1.2,2.9) node {$p_{1,8}$};
\draw[color=black] (4.2,7.1281226112193625) node {$\ell_{1}$};
\draw[color=black] (-1.8931078250762639,7.1281226112193625) node {$\ell_{2}$};
\draw[color=black] (-9.659677314993441,7.1281226112193625) node {$\ell_{3}$};
\draw[color=black] (-11.025784000516332,-1.9286587483583046) node {$\ell_{4}$};
\draw[color=black] (-4.2,-7) node {$\ell_{5}$};
\draw[color=black] (2.05,-7.) node {$\ell_{6}$};

\draw[color=black] (7.2,-4.3) node {$\ell_{7}$};
\draw[color=black] (7.5,2.1) node {$\ell_{8}$};

\draw [fill=black] (-3.385,0.005) circle (0.99mm);
\draw[color=black] (-3.4,0.5) node {$p_{1,3}$};

\draw [fill=black] (-2.375,-2.433355697996826) circle (0.99mm);
\draw[color=black] (-2.095494000709285,-2.8) node {$p_{2,4}$};

\draw [fill=black] (0.06335569799682488,-3.4433556979968256) circle (0.99mm);
\draw[color=black] (0.8,-3.4) node {$p_{3,5}$};

\draw [fill=black] (2.50171139599365,-2.4333556979968254) circle (0.99mm);
\draw[color=black] (3.2,-2.85) node {$p_{4,6}$};

\draw [fill=black] (3.51171139599365,0.005) circle (0.99mm);
\draw[color=black] (4.4,0.2) node {$p_{5,7}$};

\draw [fill=white] (0,0.005) circle (2.9mm);

\draw [fill=black] (2.5017113959936514,2.4433556979968247) circle (0.99mm);
\draw[color=black] (3.2,2.1) node {$p_{6,8}$};

\draw [fill=black] (0.0633556979968255,3.453355697996825) circle (0.99mm);
\draw[color=black] (0.2,4.1) node {$p_{1,7}$};

\draw [fill=black] (-2.375,2.4433556979968256) circle (0.99mm);
\draw[color=black] (-3.411004142323921,2.8) node {$p_{2,8}$};

\draw [fill=black] (-5.823355697996828,-2.433355697996826) circle (0.99mm);
\draw[color=black] (-6.295007145094468,-1.99) node {$p_{1,4}$};

\draw [fill=black] (-2.375,-5.88171139599365) circle (0.99mm);
\draw[color=black] (-1.6,-5.9) node {$p_{2,5}$};

\draw [fill=black] (2.50171139599365,-5.881711395993651) circle (0.99mm);
\draw[color=black] (3.2,-5.7) node {$p_{3,6}$};

\draw [fill=black] (5.950067093990474,-2.4333556979968254) circle (0.99mm);
\draw[color=black] (6.3,-1.99) node {$p_{4,7}$};

\draw [fill=black] (5.9500670939904765,2.443355697996824) circle (0.99mm);
\draw[color=black] (6.2,2.) node {$p_{5,8}$};

\draw [fill=black] (2.5017113959936514,5.89171139599365) circle (0.99mm);
\draw[color=black] (3.3,5.939103829375368) node {$p_{1,6}$};

\draw [fill=black] (-2.375,5.891711395993651) circle (0.99mm);
\draw[color=black] (-3.132723150828517,5.964402101329494) node {$p_{2,7}$};

\draw [fill=black] (-5.823355697996827,2.443355697996826) circle (0.99mm);
\draw[color=black] (-5.611953802333023,3.) node {$p_{3,8}$};

\draw [fill=black] (-5.835,5.915) circle (0.99mm);
\draw[color=black] (-5.485462442562384,6.445069268457919) node {$p_{3,7}$};

\draw [fill=black] (-9.166361011887952,-0.005990079844608537) circle (0.99mm);
\draw[color=black] (-8.82483434050723,0.5252736311920688) node {$p_{4,8}$};

\draw [fill=black] (-5.955,-6.025) circle (0.99mm);
\draw[color=black] (-5.510760714516512,-6.5) node {$p_{1,5}$};

\draw [fill=black] (0.085,-5.885) circle (0.99mm);
\draw[color=black] (0.8,-5.849890901248077) node {$p_{2,6}$};

\draw[color=blue] (-10.94988918465395,-0.4) node {$\ell_{6}'$};

\draw[color=blue] (-10.94988918465395,-4.0) node {$\ell_{5}'$};

\draw[color=blue] (-8.,-7.2) node {$\ell_{4}'$};

\draw[color=blue] (-3.25,-7.2) node {$\ell_{3}'$};

\draw[color=blue] (-0.4,-7.2) node {$\ell_{2}'$};

\draw[color=blue] (2.82,-7.2) node {$\ell_{1}'$};

\draw[color=blue] (-7.2,6.7) node {$\ell_{8}'$};
\draw[color=blue] (-10.7,3.9) node {$\ell_{7}'$};

\end{scriptsize}

\end{tikzpicture}

\end{center} 

\caption{\label{fig:Octagon}Matroid $M_8$ ($\ell_i$ and $\ell_{i+4}$ meet at infinity at point $p_{i,i+4}$).}
\end{figure}

The $8$ lines $\ell_{i},\ell_{j}$ of $\cC_{0}=(\ell_{1},\dots\ell_{8})$
meet in $28$ double points denoted by $p_{i,j}$ (some points are
at infinity). The lines $\ell_{1}',\dots,\ell_{8}'$ are the lines
containing the points in sets $S_{1},\dots,S_{8}$ which are respectively
\begin{equation}
\begin{array}{c}
\{p_{1,8},p_{2,7},p_{3,6},p_{4,5}\},\{p_{1,7},p_{2,6},p_{3,5}\},\{p_{1,6},p_{2,5},p_{3,4},p_{7,8}\},\{p_{1,5},p_{2,4},p_{6,8}\},\\
\{p_{1,4},p_{2,3},p_{5,8},p_{6,7}\},\{p_{1,3},p_{4,8},p_{5,7}\},\{p_{1,2},p_{3,8},p_{4,7},p_{5,6}\},\{p_{2,8},p_{3,7},p_{4,6}\}.
\end{array}\label{eq:Octa-Set-po}
\end{equation}
These sets $S_{k},\,k=1,\dots,8$ form a partition of the $28$ double
points of $\cC_{0}$; these $28$ points are the triple points of
$\cC_{0}\cup\cC_{1}$. One has the relation $\L_{\{2\},\{3,4\}}(\cC_{0})=\cC_{1}$
(as unlabeled line arrangements).

Let $M_{8}$ be the matroid associated to the incidences between the
$16$ labeled lines $\ell_{1},\dots,\ell_{8},\ell_{1}',\dots,\ell_{8}'$
and the $28$ triple points: it is obtained from the matroid associated
to the labeled line arrangement $\cC_{0}\cup\cC_{1}$, but we discard
all non-bases coming from the central point, so that $M_{8}$
has $16$ atoms and only $28$ non-bases. We denote by $\cR_{8}$
the moduli space of realizations of $M_{8}$ (over $\CC)$. 
\begin{rem}
\label{rem:No-cano-label}A priori, there is no canonical choice for
the labelings of the lines $\ell_{1}',\dots,\ell_{8}'$ in the unlabeled
line arrangement $\L_{\{2\},\{3,4\}}(\cC_{0})$. The choice we made
in Equation \eqref{eq:Octa-Set-po} will be justified later, see Remark
\ref{rem:No-Cano-label-2}.
\end{rem}

\subsection{The moduli space $\protect\cR_{8}$ of $M_{8}$}

A direct computation in \texttt{OSCAR} shows that the moduli space $\cR_{8}$ is two-dimensional,
and an open sub-set of the quartic surface $Z_{8}$ in $\PP^{3}$
with the equation
\[
y_{1}y_{2}^{2}y_{3}-y_{1}^{2}y_{2}y_{4}+y_{1}y_{2}^{2}y_{4}+y_{1}^{2}y_{3}y_{4}-2y_{1}y_{2}y_{3}y_{4}-y_{1}y_{3}^{2}y_{4}+y_{1}y_{3}y_{4}^{2}-y_{2}y_{3}y_{4}^{2}+y_{3}^{2}y_{4}^{2}=0.
\]
The surface $Z_{8}$ has singularities $A_{2},A_{2},A_{3},A_{4},A_{3},A_{1}$
at the respective points {
\[
(1:0:0:0):(0:1:0:0):(0:0:1:0):(0:0:0:1):(1:1:1:1):(1:0:1:0).
\]
}Its minimal desingularization $Z_{8}^{s}$ is a K3 surface. The realization
$\cA(x)$ corresponding to a generic point $x=(x_{1},x_{2},x_{3})$
of $Z_{8}$ in the affine chart $\mathbb{A}^{3}=\{y_{4}\neq0\}$ is
the union $\cA(x)=\cC_{0}(x)\cup\cC_{1}(x)$, where $\cC_{0}(x)$
is the line arrangement with eight lines with normal vectors the four
vectors of the canonical basis and the following four vectors {
\[
\begin{array}{c}
(x_{1}-x_{2}:x_{1}^{2}-x_{1}x_{2}-x_{1}x_{3}+x_{1}-x_{2}+x_{3}:x_{1}-x_{2}x_{3}-x_{2}+x_{3}),\\
(x_{1}x_{2}-x_{1}x_{3}-x_{2}+x_{3}:x_{1}x_{2}^{2}-x_{1}x_{2}-x_{1}x_{3}+x_{1}-x_{2}+x_{3}:x_{1}x_{2}x_{3}-2x_{1}x_{3}+x_{1}-x_{2}+x_{3})\\
(x_{1}-1:x_{1}x_{2}-x_{2}:x_{1}-x_{2}),(1:x_{1}:x_{3}).
\end{array}
\]
}Moreover, $\cC_{1}(x)$ is the line arrangement with normal vectors
\[
\begin{array}{c}
(x_{1}x_{2}-x_{1}x_{3}-x_{2}+x_{3}:x_{1}x_{2}^{2}-x_{1}x_{2}-x_{1}x_{3}+x_{1}-x_{2}+x_{3}:x_{1}x_{2}-x_{1}x_{3}-x_{2}^{2}+x_{2}x_{3}),\\
(x_{1}^{2}x_{2}-x_{1}^{2}x_{3}-x_{1}x_{2}^{2}+x_{1}x_{2}x_{3}-x_{1}x_{2}+x_{1}x_{3}+x_{2}^{2}-x_{2}x_{3}:x_{1}^{3}x_{2}-x_{1}^{3}x_{3}-x_{1}^{2}x_{2}^{2}+x_{1}^{2}x_{3}^{2}\\
+2x_{1}x_{2}x_{3}-x_{1}x_{2}-2x_{1}x_{3}^{2}+x_{1}x_{3}+x_{2}^{2}-2x_{2}x_{3}+x_{3}^{2}:x_{1}^{2}x_{2}x_{3}-x_{1}^{2}x_{2}-x_{1}^{2}\\
x_{3}+x_{1}^{2}+x_{1}x_{2}^{2}-2x_{1}x_{2}-x_{1}x_{3}^{2}+2x_{1}x_{3}+x_{2}^{2}-2x_{2}x_{3}+x_{3}^{2}),\\
(x_{1}-x_{2}:x_{1}-x_{2}:x_{1}-x_{2}x_{3}-x_{2}+x_{3}),(x_{3}:x_{1}x_{2}:x_{3}),(0:1:1),\\
(x_{1}-1:0:x_{1}-x_{3}),(1:x_{1}:0),(1:x_{2}:x_{3}).
\end{array}
\]
From the definition of the matroid $M_{8}$, if $\cA_{0}=\cC_{0}\cup\cC_{1}$
is a realization of $M_{8}$ and $\cC_{0}$ (resp. $\cC_{1}$) denotes
its first (resp. last) eight lines then, $\L_{\{2\},\{3,4\}}(\cC_{0})=\cC_{1}$
as unlabeled line arrangements. The following operator $\L_{\{2\},\{3,4\}}^{\ell}$
gives a labeling to $\L_{\{2\},\{3,4\}}(\cC_{0})$:
\begin{defn}
The operator $\L_{\{2\},\{3,4\}}^{\ell}$ associates to a labeled
line arrangement $L_{8}$ of $8$ lines $\ell_{1},\dots,\ell_{8}$,
the labeled line arrangement $\ell_{1}',\dots,\ell_{8}'$ where $\ell_{j}'$
is the set of lines containing all the points in $S_{k}$ defined
in \eqref{eq:Octa-Set-po} ($\ell_{j}'$ is a line or the empty set). 
\end{defn}

For a generic arrangement $L_{8}$ of eight lines, one has $\L_{\{2\},\{3,4\}}^{\ell}(L_{8})=\emptyset$.
The operator $\L_{\{2\},\{3,4\}}^{\ell}$ is constructed so that if
$\cA_{0}$ is any realization of $M_{8}$ and $\cC_{0}$ (resp. $\cC_{1}$)
denotes its first (resp. last) eight lines then $\L_{\{2\},\{3,4\}}^{\ell}(\cC_{0})=\cC_{1}$
as labeled line arrangements (and of course $\L_{\{2\},\{3,4\}}(\cC_{0})=\cC_{1}$
if one forgets the labels). 

\subsection{The operator $\protect\L_{\{2\},\{3,4\}}^{\ell}(\protect\cC_{1})$
acts as a rational self-map on $\protect\cR_{8}$}

A priori the line arrangement $\L_{\{2\},\{3,4\}}^{\ell}(\cC_{1})$
could be empty, however:
\begin{thm}
Suppose that $\cA_{0}=\cC_{0}\cup\cC_{1}$ is generic among the realizations
of $M_{8}$. Then the labeled line arrangement $\cC_{2}=\L_{\{2\},\{3,4\}}^{\ell}(\cC_{1})$
has $8$ lines and $\cA_{1}=\cC_{1}\cup\cC_{2}$ is a realization
of $\cR_{8}$.
\end{thm}

\begin{proof}
Using the function field of $\cR_{8}$, we realize the generic element
of $\cR_{8}$ using the formulas for $\cA(x)$.
Then we compute $\cC_{2}=\L_{\{2\},\{3,4\}}^{\ell}(\cC_{1})$
and obtain eight lines.
Finally we check that $\cC_{1}\cup\cC_{2}$ defines
the same matroid as $\cA_{0}$. 
\end{proof}
The operator $\L_{\{2\},\{3,4\}}^{\ell}$ acts on realizations of
$M_{8}$, sending $\cA_{0}=\cC_{0}\cup\cC_{1}$ to $\cA_{1}=\cC_{1}\cup\cC_{2}$,
where $\cC_{2}=\L_{\{2\},\{3,4\}}^{\ell}(\cC_{1})$. It therefore
acts on the moduli space $Z_{8}$: we denote by 
\[
\l_{\{2\},\{3,4\}}:Z_{8}\dashrightarrow Z_{8}
\]
that action. In order to obtain the explicit polynomials defining
$\l_{\{2\},\{3,4\}}$, we remark that one may recover the coordinates
$x_{1},x_{2},x_{3}$ of the line arrangement $\cA_{0}(x)$ from the
two last normal vectors $(1:x_{1}:0),(1:x_{2}:x_{3})$ of $\cC_{1}(x)$.
Then one computes the unique line arrangement $\tilde{\cC_{1}}\cup\tilde{\cC_{2}}$
projectively equivalent to $\cA_{1}(x)=\cC_{1}(x)\cup\cC_{2}(x)$
such that the first four normal vectors are the canonical basis. The
image of $x$ by $\l_{\{2\},\{3,4\}}$ is the point $x'=(x_{1}',x_{2}',x_{3}')$
such that the two last normal vectors of $\tilde{\cC_{2}}$ are $(1:x_{1}':0),(1:x_{2}':x_{3}')$
(and $\tilde{\cC_{1}}\cup\tilde{\cC_{2}}=\cA_{0}(x')$). Taking the
homogenization to $\PP^{3}$, one obtains that the map $\l_{\{2\},\{3,4\}}$
is defined by the four degree $10$ coprime polynomials $P_{1},\dots,P_{4}$
given in the ancillary file of the arXiv version of this paper. The base
points of $\l_{\{2\},\{3,4\}}$ are
\[
\begin{array}{c}
(-\sqrt{2}-1:\sqrt{2}+2:2\sqrt{2}+3:1),(\sqrt{2}-1:-\sqrt{2}+2:-2\sqrt{2}+3:1),\\
(i:0:1:1),(-i:0:1:1),(1:1:0:1),(0:1:1:0),(0:1:0:1).
\end{array}
\]
The line arrangements $\cC_{0}\cup\cC_{1}$ associated to the first
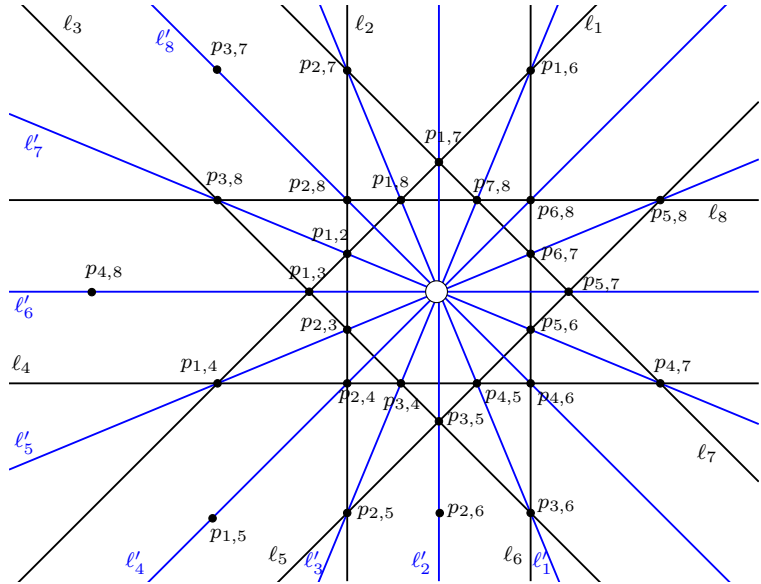
two points are the regular octagon and its lines of symmetries. The
line arrangements $\cC_{0}$ associated to the third and fourth points
are such that $\L_{\{2\},\{3,4\}}(\cC_{0})$ is the Ceva line arrangement
with $12$ lines; it contains $\cC_{0}$. 

Using the explicit polynomials $P_{1},\dots,P_{4}$, we obtain that:
\begin{prop}
The degree of the rational self-map $\l_{\{2\},\{3,4\}}$ on $Z_8^{s}$
is $4$.
\end{prop}

\begin{proof}
We again apply the method from \cite{Voisin}. Let $f(x_{1},x_{2},x_{3})$ be
the equation of the quartic $Z_8$ in the chart $U_{4}:y_{4}\neq0$.
The space of global non-vanishing differential $2$-forms is generated
by a form~$\o$, which one can choose so that on an open set of $U_{4}$
one has: 
$
\omega=\frac{dx_{2}\wedge dx_{3}}{\partial f/\partial x_{1}}.
$
The rational self-map $\l_{\{2\},\{3,4\}}$ preserves $U_{4}$.  A direct computation gives that 
$
\l_{\{2\},\{3,4\}}^{*}\o=-2\o.
$
The pull-back by $\l_{\{2\},\{3,4\}}$ of the volume form $\om\bar{\om}$
is therefore $4\om\bar{\om}$, thus the degree of $\l_{\{2\},\{3,4\}}$
is $4$. 
\end{proof}

\subsection{The dynamical system $(Z_{8},\protect\l_{\{2\},\{3,4\}})$ is semi-conjugated to the plane}

The four polynomials $P_{1},\dots,P_{4}$ such that 
$\l_{\{2\},\{3,4\}}=(P_{1}:\dots:P_{4})$
 verify $\deg_{y_{1}}(P_{1})=1$ and $\deg_{y_{k}}(P_{1})=0$
for $k\geq2$.
Let $\pi:Z_{8}\dashrightarrow\PP^{2}$ be the double cover
obtained by projecting from the double point $(1:0:0:0)$ of $Z_{8}$. 
\begin{lem}
The branch curve $B$ of $\pi$ is the union of the conic $C=\{z_{1}^{2}-z_{2}z_{3}=0\}$
and the quartic curve 
\[
Q=\{z_{1}^{2}z_{2}^{2}+2z_{1}^{2}z_{2}z_{3}-4z_{1}z_{2}^{2}z_{3}-z_{2}^{3}z_{3}+z_{1}^{2}z_{3}^{2}-4z_{1}z_{2}z_{3}^{2}+6z_{2}^{2}z_{3}^{2}-z_{2}z_{3}^{3}\}.
\]
\end{lem}

\begin{proof}
The ramification locus of $\pi_{1}$ is the discriminant of the equation
of $Z_8$ with respect to the variable $y_{1}$. The image of the ramification
curve by $\pi_{1}$ is the curve $B$.
\end{proof}
The quartic $Q$ has geometric genus $0$ and is singular at the points
$(1:0:0),(1:1:1)$ with singularities $A_{3}$ and $A_{1}$. The curve
$B=C+Q$ is singular at the points 
\[
(1:0:0),(0:1:0),(0:0:1),(1:1:1)
\]
 with singularities $A_{3},A_{5},A_{5},D_{4}$, respectively.

Let us define the polynomials $Q_{k}=P_{k+1}(0,z_{1},z_{2},z_{3})$
$(k=1,2,3$) and the rational self-map $\mu:\PP^{2}\dashrightarrow\PP^{2}$,
$\mu=(Q_{1}:Q_{2}:Q_{3})$. One has 
$
\mu^{*}(B)=B+2D
$
for a degree $27$ curve $D$. Using Subsection \ref{subsecTheoretical-SConj},
the double cover of $\PP^{2}$ branched
over $B=C+Q$ is birational to the surface $Z_{8}$ and $(Z_{8},\l_{\{2\},\{3,4\}})$
is semi-conjugated to $(\PP^{2},\mu)$.

The indetermination points of $\mu$ are the $9$ points
\[
\begin{array}{c}
(1:0:0),(0:1:0),(0:0:1),(1:1:1),(1:0:1),(1:1:0),\\
(0:1:1),(-\sqrt{2}+2:-2\sqrt{2}+3:1),(\sqrt{2}+2:2\sqrt{2}+3:1).
\end{array}
\]
The image by $\mu$ of $Q$ is the conic $C$; the rational map $\mu$
restricts to the identity on $C$.

\begin{rem}
\label{rem:No-Cano-label-2}
The choice for the labelings of the lines in the unlabeled
line arrangement $\L_{\{2\},\{3,4\}}(\cC_{0})$ was made
 so that the defining polynomials of the rational self-map $\l_{\{2\},\{3,4\}}$
are of low degree. Moreover, for the other choices we tried, the degrees
of the polynomials defining the analog of $\l_{\{2\},\{3,4\}}$ with
respect to any variables $y_{i}$ were never $1,0,0,0$, so that it
was not possible to understand that rational self-map $\l_{\{2\},\{3,4\}}$
as a semi-conjugacy with the plane. 
\end{rem}

\subsection{The K3 surface $Z_{8}$ and the modular surface $\Xi_{1}(8)$}

One has:
\begin{prop}
The K3 surface $Z_8^{s}$ is the unique K3 surface with discriminant
$-8$ and Picard number $20$.
\end{prop}

\begin{proof}
The eight lines with equations
\[
\begin{array}{c}
(y_{1}=y_{3}=0),(y_{1}=y_{4}=0),(y_{2}=y_{3}=0),(y_{2}=y_{4}=0),(y_{3}=y_{4}=0),\\
(y_{1}-y_{4}=y_{2}-y_{4}=0),(y_{1}-y_{3}=y_{2}-y_{4}=0),(y_{2}-y_{4}=y_{3}-y_{4}=0)
\end{array}
\]
are contained in the surface $Z_8$. Using Magma, one can compute that their
strict transforms on $Z_8^{s}$ together with the $15$ $(-2)$-curves
coming from the resolutions of the singularities of $Z_8$, generate
a rank $20$ lattice with discriminant $-8$. There is no K3 surface
with Picard number $20$ and discriminant $-2$ and there is a unique K3 surface
with Picard number 20 and discriminant -8 (see e.g. \cite{Schuett}) which yields the conclusion. 
\end{proof}
\begin{prop}\label{ModSur8}
The surface $Z_8$ is (isomorphic to) the elliptic modular surface $\Xi_{1}(8)$
above the modular curve $X_{1}(8)$.
\end{prop}

\begin{proof}
The projection map from the line $y_{2}-y_{4}=y_{3}-y_{4}=0$ induces
a fibration $Z_8\to\PP^{1}$. By evaluating the Equation of $Z_8$ at
$(X,1+t(Y-1),Y,1)$, one gets the cubic affine model 
\[
(t-1)X^{2}-t^{2}XY^{2}+XY+(t-1)^{2}X+(t-1)Y=0
\]
of the generic fiber, where $t$ is the parameter of $\PP^{1}$. One
computes that the Weierstrass model of it is the elliptic curve 
\[
\begin{array}{c}
E:y^{2}=x^{3}+(4t^{4}-8t^{3}+4t^{2}+1)/t^{4}x^{2}+8(t-1)^{2}/t^{6}x+16(t-1)^{4}/t^{8}.\end{array}
\]
The associated elliptic surface is a smooth model of the K3 surface
$Z_8$: it is isomorphic to $Z_8^{s}$. One computes that the singular
fibers of the fibration are 
$
2I_{8}+I_{4}+I_{2}+2I_{1},
$
at the points $1,0,\infty,1/2,t^{2}-t-1/4=0$, respectively.

By \cite[Section 2.3.3]{TY}, the equation of a Weierstrass model
of the elliptic surface $\Xi_{1}(8)$ above the modular curve $X_{1}(8)$
is
\[
E':\,\,\eta^{2}=\xi^{3}+(2-s^{2})\xi^{2}+\xi,
\]
where $s=2t^{2}/(t^{2}-1)$. To check that $\Xi_{1}(8)$ is isomorphic
to $Z_8^{s}$, one just has to compare the two $j$-invariants $j(E)(t)\in\QQ(t)$
and $j(E')(t)\in\QQ(t)$. We compute that $j(E)(\tfrac{1}{2}(1-\tfrac{1}{t}))=j(E')(t)$,
therefore $E$ is isomorphic to $E'$, and $\Xi_{1}(8)\simeq Z_8^{s}$. 
\end{proof}

\subsection{Action of $\protect\aut(M_{8})$}

The automorphism group of $M_{8}$ is generated by the involutions 
\[
\begin{array}{c}
s_{1}=(2,4)(3,7)(6,8)(9,11)(10,14)(13,15),s_{2}=(2,6)(4,8)(9,13)(11,15),\\
s_{3}=(1,2)(3,8)(4,7)(5,6)(9,13)(10,12)(14,16).
\end{array}
\]
The group $\aut(M_{8})$ is the semi-direct product $\ZZ/8\ZZ\rtimes (\ZZ/2\ZZ)^{2}$.
One computes that it acts faithfully on the K3 surface $Z_{8}$. The
map $s_{2}$ (acting on $Z_8$) is given in the ancillary file of 
the arXiv version of this paper. It is a birational
involution of $\PP^{3}$. 

The group of elements $\s$ commuting with the action of $\l_{\{2\},\{3,4\}}$
is isomorphic to $(\ZZ/2\ZZ)^{2}$. The involution $s=(1,5)(2,6)(3,7)(4,8)$
is the unique automorphism of $\aut(M_{8})$ such that $\l_{\{2\},\{3,4\}}\circ s=\l_{\{2\},\{3,4\}}.$

\subsection{Periodic line arrangements}

Let us prove:
\begin{prop} \label{prop22}
The surface  $Z_{8}$ contains a curve $C_{3}$ of geometric genus $5$
such that each point of $C_{3}$ is fixed by $\l_{\{2\},\{3,4\}}$
and for a generic point $x$ of $C_{3}$, the associated line arrangement $\cC_{0}(x)$
in $\PP^{2}$ is periodic of period $3$ for the action of $\L_{\{2\},\{3,4\}}^{\ell}$.
\end{prop}

\begin{rem}
We recall that $\L_{\{2\},\{3,4\}}$ is an operator acting on line arrangements,
whereas $\l_{\{2\},\{3,4\}}$ is the rational self-map induced by $\L_{\{2\},\{3,4\}}$: 
it acts on line arrangements modulo
projective transformations.
In particular, Proposition \ref{prop22} implies that 
for a line arrangement $\cC$ corresponding to a point on the curve $C_3$, one has
$(\L_{\{2\},\{3,4\}})^{\circ 3}(\cC)=\cC$ with $\L_{\{2\},\{3,4\}}(\cC)\neq \cC$, but 
$\L_{\{2\},\{3,4\}}(\cC)$ is projectively equivalent to $\cC$. The union of the line arrangements 
$(\L_{\{2\},\{3,4\}})^{\circ k}(\cC),\,k=0,1,2$ has 24 lines with $84$ triple points, $24$ double points, and no other singularities.
\end{rem}

\begin{proof}
We searched by random an example of a $\l$-fixed  
point $x$ over a finite field and we found the point $x=(794:582:116:1)\in\PP^{3}(\FF_{1013})$ in the surface $(\cR_{8})_{/\FF_{1013}}$.
The corresponding line arrangement $\cC_{0}\cup\cC_{1}$ is $3$-periodic for the operator $\L_{\{2\},\{3,4\}}^{\ell}$:
 the line arrangements $\cC_{0}\cup\cC_{1}$, 
$\cC_{1}\cup\cC_{2}$ and $\cC_{2}\cup\cC_{0}$ are realizations of $M_8$ (over $\FF_{1013}$), and 
$\cC_{k+1\,mod\,3}=\L_{\{2\},\{3,4\}}^{\ell}(\cC_k)$. 
One computes that the matroid $N_{24}$ associated to $\cC_{0}\cup\cC_{1}\cup\cC_{2}$
has an irreducible one dimensional moduli space $\cR(N_{24})$ over $\CC$ and that
 the geometric genus of the compactification of $\cR(N_{24})$ is $5$.
Let $\cC_{0}'\cup\cC_{1}'\cup\cC_{2}'$ be a realization (over $\CC$) of $N_{24}$. 
From the combinatorics of $M_8$ and $N_{24}$, the line arrangements $\cC_{0}'\cup\cC_{1}'$, 
$\cC_{1}'\cup\cC_{2}'$ and $\cC_{2}'\cup\cC_{0}'$ are realizations of $M_8$, and 
$\cC_{k+1\,mod\,3}'=\L_{\{2\},\{3,4\}}^{\ell}(\cC_k')$ if the realization is generic. 
The natural
map $\cR(N_{24})\to\cR_{8}$, which to a realization $\cC_{0}'\cup\cC_{1}'\cup\cC_{2}'$
of $N_{24}$ associates $\cC_{0}'\cup\cC_{1}'$ is one-to-one
onto its image (a curve denoted $C_{3}$) in $\cR_8$, since one may recover $\cC_{2}'$ (and therefore
$\cC_{0}'\cup\cC_{1}'\cup\cC_{2}'$) as $\cC_{2}'=\L_{\{2\},\{3,4\}}^{\ell}(\cC_{1}')$. 
A computer computation gives that $C_3$ has genus $5$ and 
$\Lambda_{\{2\},\{3,4\}}(\eta)$ is projectively equivalent to $\eta$,
 where $\eta$ is the generic point of $C_3$, thus any specialization 
 $\eta'$ is such that $\l_{\{2\},\{3,4\}}(\eta')=\eta'$, and the curve $C_3$ is point-wise fixed by $\l_{\{2\},\{3,4\}}$.
\end{proof}

\end{document}